\pdfoutput=1
\documentclass[11pt]{article} 

\usepackage[utf8]{inputenc} 
\usepackage[british]{babel} 

\begin{hyphenrules}{british}
\end{hyphenrules}

\usepackage{geometry} 
\usepackage{graphicx} 

\usepackage[parfill]{parskip} 

\usepackage{booktabs} 
\usepackage{array} 
\usepackage{paralist} 
\usepackage{verbatim} 
\usepackage{subfig} 
\usepackage{bmpsize}
\usepackage{amssymb}
\usepackage{mathtools}
\usepackage{amsmath}	
\usepackage{amsthm}
\usepackage{graphicx}
\usepackage{faktor} 
\usepackage{mathrsfs} 
\usepackage[all, cmtip]{xy} 
\usepackage{multicol} 
\usepackage{diagbox} 
\usepackage{multirow} 
\usepackage{tikz} 
\usepackage{array, makecell, rotating}
\usepackage{longtable} 
\usepackage{caption} 

\setcellgapes{2pt}
\newcommand\nocell[1]{\multicolumn{#1}{c|}{}}

\makeatletter
\renewcommand*\env@matrix[1][\arraystretch]{%
  \edef\arraystretch{#1}%
  \hskip -\arraycolsep
  \let\@ifnextchar\new@ifnextchar
  \array{*\c@MaxMatrixCols c}}
\makeatother  

\numberwithin{equation}{subsection}

\newcommand\Tstrut{\rule{0pt}{3.5ex}}       
\newcommand\Bstrut{\rule[-1.4ex]{0pt}{0pt}} 
\newcommand\TBstrut{\Tstrut\Bstrut}         
\newcommand\TTstrut{\rule{0pt}{6ex}}       

\allowdisplaybreaks 

\newtheorem{theorem}{Theorem}[subsection]
\newtheorem{thm}{Theorem}
\theoremstyle{remark}
\newtheorem{definition}[theorem]{Definition}
\theoremstyle{theorem}
\newtheorem{proposition}[theorem]{Proposition}

\newtheorem{corollary}[theorem]{Corollary}
\theoremstyle{remark}
\newtheorem{remark}[theorem]{Remark}
\theoremstyle{remark}
\newtheorem{example}[theorem]{Example}

\newtheorem*{notation}{Notation}
\usepackage{fancyhdr} 
\usepackage{sectsty}
\allsectionsfont{\sffamily\mdseries\upshape} 

\usepackage[nottoc,notlof,notlot]{tocbibind} 
\usepackage[titles,subfigure]{tocloft} 

\title{K3 surfaces with a symplectic automorphism of order 4}
\author{Benedetta Piroddi}

\begin{document}
\maketitle
\begin{abstract}
Given $X$ a K3 surface admitting a symplectic automorphism $\tau$ of order 4, we describe the isometry $\tau^*$ on $H^2(X,\mathbb Z)$. Having called $\tilde Z$ and $\tilde Y$ respectively the minimal resolutions of the quotient surfaces $Z=X/\tau^2$ and $Y=X/\tau$, we also describe the maps induced in cohomology by the rational quotient maps $X\rightarrow\tilde Z,\ X\rightarrow\tilde Y$ and $\tilde Y\rightarrow\tilde Z$: with this knowledge, we are able to give a lattice-theoretic characterization of $\tilde Z$, and find the relation between the Néron-Severi lattices of $X,\tilde Z$ and $\tilde Y$ in the projective case. We also produce three different projective models for $X,\tilde Z$ and $\tilde Y$, each associated to a different polarization of degree 4 on $X$.
\end{abstract}

\section*{Introduction} 
An automorphism $\alpha$ of a K3 surface $S$ is symplectic if it preserves its volume form: therefore, the surface $T$ that is the minimal resolution of the singularities of the quotient $S/\alpha$ is again a K3 surface.
Nikulin characterized the K3 surfaces $S$ that admit a symplectic automorphism of order $n$ 
by the existence of a primitive embedding of a certain lattice $\Omega_n$ in $NS(S)$, and K3 surfaces $T$
by the existence of a primitive embedding of another lattice $M_n$ in $NS(T)$ (\cite{Nikulin2}). The first explicit description of the map that relates the lattices $H^2(S,\mathbb Z)$ and $H^2(T,\mathbb Z)$ was given by Morrison (\cite{Morrison}) for a symplectic involution, and subsequent works by Van Geemen, Garbagnati and Sarti (\cite{VGS}, \cite{GS}) produced a complete description of the correspondence between families of projective K3 surfaces that admit a symplectic involution, and those that arise as desingularization of their quotient. A similar approach was used by Garbagnati and Prieto (\cite{GP})
for symplectic automorphisms of order 3, and will be used in this work too.\\
After an exposition of the necessary results of lattice theory in Section 1, in Section 2 we select a K3 surface with a large Néron-Severi lattice and a symplectic automorphism $\tau$ of order 4, from which we deduce in Section 3 the action of the isometry $\tau^*$ on the lattice $H^2(X,\mathbb Z)$ for any K3 surface $X$ (by \cite{Nikulin2}, Thm. 4.7). Section 4 is devoted to the description of the maps induced in cohomology by the rational maps $\xymatrix{\pi_2: X \ar@{-->}[r] &\tilde Z}$, $\xymatrix{\pi_4: X \ar@{-->}[r] &\tilde Y}$, $\xymatrix{\widehat{\pi_2}:\tilde Z \ar@{-->}[r] &\tilde Y}$, where $\tilde Z$ and $\tilde Y$ are the minimal resolutions of the quotient surfaces $Z=X/\tau^2$ and $Y=X/\tau$ respectively. In Section 5, a lattice-theoretic approach is used to determine the families of projective K3 surfaces $X$ with an action of $\tau$: the Néron-Severi of $X$ has in fact to admit a primitive embedding of both the lattice $\Omega_4$, as described by Nikulin, and the lattice of rank one $\langle 2d\rangle$ for some positive integer $d$, generated by an ample class $L$; depending on the choice of $d$, the lattices of minimal rank $rk(\Omega_n)+1$ admissible as Néron-Severi of $X$ can be one or more, giving rise to many lattice-polarized families; the correspondence between families of $X$, $\tilde Z$ and $\tilde Y$ is given via the maps introduced in the previous section. In Section 6 we describe the action induced by $\tau$ on the projective space $\mathbb P(H^0(X,L)^*)$ for each of the existing families of $X$, and as an example we find the explicit projective models of $X$, $\tilde Z$ and $\tilde Y$ associated to the three different families with $L^2=4$.

The main results presented in this work are the following:
\begin{thm}[see Thm. \ref{Znonproj}]
A K3 surface $\tilde Z$ is the minimal resolution of $X/\tau^2$, for some K3 surface $X$ with a symplectic automorphism of order 4 $\tau$, if and only if $Z$ is $\Gamma$-polarized (the lattice $\Gamma$ is defined in Def. \ref{def:Gamma}).
\end{thm}
\begin{thm}[see Thm. \ref{relations},\ \ref{relationsZ}]
Let $X$ be a projective K3 surface with a symplectic automorphism of order 4 $\tau$, let $\tilde Z$ and $\tilde Y$ be respectively the minimal resolution of $X/\tau^2$ and $X/\tau$. Then, using the notation introduced in Section \ref{sec:proj}, we have the following correspondence between $NS(X),\ NS(\tilde Z)$ and $NS(\tilde Y)$:
\begin{longtable}{|c|c|c|c|}
\cline{2-4}
\nocell{1} &{$NS(X)$\Tstrut} 				                     &{$NS(\tilde Z)$ \Tstrut}  &{$NS(\tilde Y)$ \Tstrut}\\ [6pt]
\hline
$\forall d$\Tstrut	& $\Omega_4\oplus\langle 2d\rangle$\Tstrut  
		  	& $(\Gamma\oplus\langle 4d\rangle)' $ & $(M\oplus\langle 8d\rangle)^{\star}$ \\ [6pt]
\hline
\multirow{2}{*}{$d=_4 2$\Tstrut} & $(\Omega_4\oplus\langle 2d\rangle)'^{(1)}$\Tstrut 
& \multirow{2}{*}{$(\Gamma\oplus\langle 4d\rangle)^\star$\Tstrut} & $(M\oplus\langle 2d\rangle)'^{(1)}$ \\[6pt]
\ & $(\Omega_4\oplus\langle 2d\rangle)'^{(2)}$\Tstrut &\   &$(M\oplus\langle 2d\rangle)'^{(2)} $  \\ [6pt]
\hline
$d=_4 3$\Tstrut & $(\Omega_4\oplus\langle 2d\rangle)'$\Tstrut  
& $(\Gamma\oplus\langle 4d\rangle)^\star $  & $(M\oplus\langle 2d\rangle)'$ \\ [6pt]
\hline
\multirow{2}{*}{$d=_4 0$\Tstrut} & $(\Omega_4\oplus\langle 2d\rangle)'$\Tstrut 
& \multirow{2}{*}{$(\Gamma\oplus\langle d\rangle)' $\Tstrut}  &$(M\oplus\langle 2d\rangle)'$ \\[6pt]
\ & $(\Omega_4\oplus\langle 2d\rangle)^{\star}$\Tstrut  &\ &$M\oplus\langle d/2\rangle $   \\ [6pt]
\hline
\end{longtable}
\end{thm}

\section{Lattices}\label{sec:lattices}
In this section, we're going to recall some fundamental results on lattices and discriminant forms; most of them are due to Nikulin, and are exposed in (\cite{Nikulin1}, \S 1).
\begin{definition}
An \emph{even lattice} is a free $\mathbb Z$-module $S$ of finite rank, equipped with a nondegenerate quadratic form $q: S\rightarrow 
2\mathbb Z$. Working in characteristic different than two, this is equivalent to giving an integral nondegenerate bilinear symmetric even form $b:S\rightarrow\mathbb Z$; we will refer to $b$ as \emph{intersection form} of $S$.\\
An \emph{isomorphism} between lattices (or \emph{isometry}) is an isomorphism of $\mathbb Z$-modules that preserves the intersection form. Denote $O(S)$ the group of isometries of $S$ into itself.
\end{definition}
\begin{definition}
Define the \textit{K3 Lattice} 
\[\Lambda_{\mathrm{K3}}\simeq E_8^{\oplus 2} \oplus U^{\oplus 3},\]
where $E_8$ is the unique even negative definite unimodular lattice of rank 8, and $U$ is the unique even indefinite unimodular lattice of rank 2; the K3 lattice is isometric to the second integral cohomology group $H^2(X,\mathbb{Z})$ equipped with the cup product ${H^2(X,\mathbb{Z})\times H^2(X,\mathbb{Z})\rightarrow H^4(X,\mathbb{Z})\simeq\mathbb Z}$, for any K3 surface $X$.
\end{definition}

\begin{definition}
Let $S$ be an even lattice: define the \emph{dual lattice} $S^*=\{\ x \in S\otimes \mathbb Q \ \mid\ \forall s \in S,\ b_{\mathbb Q}(x, s) \in\mathbb Z\ \}$
where $b_{\mathbb Q}$ denotes the $\mathbb Q$-linear extension of $b$.\\
Denote \emph{discriminant group} of $S$ $A_S\coloneqq S^*/S$, where $S\hookrightarrow S^*$ via $s\mapsto b(s, - )$. 
An invariant of the discriminant group is its \emph{length} $\ell$, that is defined as the minimum number of generators of $A_S$; we are going to write \[\lambda(S)=\lambda(A_S) =(n_1 , n_2 ,\dots , n_k)\] if the first of the generators in a set that satisfies the minimum has order $n_1$, the second $n_2$ and so on, with ${n_1\leq n_2\leq,\dots ,\leq n_k}$.\\
Define the \emph{discriminant (quadratic) form} $q_S: A_S\rightarrow\mathbb{Q}/2\mathbb{Z}$, induced on $A_S$ by the quadratic form $q$ of $S$. A subgroup $H\subset A_S$ is said to be \emph{isotropic} if it is annichilated by the discriminant form $q_S$.
\end{definition}

\subsection{Overlattices and primitive embeddings}

\begin{definition}
An \emph{embedding} of (even) lattices $(S,q)\hookrightarrow (M,\tilde q)$ is an injective homomorphism of $\mathbb{Z}$-modules such that $\tilde q|_S=q$. In this case, we say that $M$ is an \emph{overlattice} of $S$. An embedding is \emph{primitive} if $M/S$ is free; an overlattice is \emph{of finite index} if $M/S$ is a finite (abelian) group, and it is a \emph{cyclic overlattice} if $M/S$ is cyclic.
\end{definition}

\begin{theorem}[\cite{Nikulin1}, Prop. 1.4.1.a)]\label{Nik1}
Let $S$ be an even lattice, let $M$ be an overlattice of finite index of $S$, let $H_{M}=M/S$.
The correspondence $M\rightarrow H_{M}$ determines a bijection between overlattices of finite index of $S$ and isotropic subgroups of $A_S$.
\end{theorem}

\begin{definition}
Two embeddings $S\hookrightarrow M$,  $S\hookrightarrow M'$ are \emph{isomorphic} if there is an isometry between $M$ and $M'$ that restricted to $S$ is the identity.\\
Two overlattices of finite index of $S$, $Q$ and $Q'$ are \emph{isomorphic} if there is an isometry $\alpha$ of $S$ into itself that extends $\mathbb Q$-linearly to an isometry $\overline\alpha$ between $Q$ and $Q'$.
\end{definition}

\begin{theorem}[\cite{Nikulin1}, Prop. 1.6.1] \label{thm:overlattices}
A primitive embedding of an even lattice $S$ into an even unimodular
lattice $L$, in which the orthogonal complement of $S$ is isomorphic to $K$, is determined by an isomorphism $\gamma:A_S\xrightarrow{\sim} A_K$ for which the discriminant quadratic forms satisfy $q_K\circ \gamma = -q_S$.
Two such isomorphisms $\gamma$ and $\gamma'$ determine isomorphic primitive embeddings if and only if they are conjugate via an automorphism of $K$.
\end{theorem}

\begin{corollary}\label{corollario}
Let $L$, $S$ and $K$ be as in Theorem \ref{thm:overlattices}. The isomorphism classes of overlattices $Q$ of $S\oplus K$ in $L$, such that $Q/(S\oplus K)$ is cyclic, are in bijection with the isometry classes for the action of $O(S)$  induced on $A_S$ (equivalently on $A_K$ via $\gamma$).
\end{corollary}
\begin{proof}
Using the notation of the previous theorem, fix the isomorphism $\gamma: A_S\simeq A_K$; let $s\in A_S$ such that $q_S(s)=d\in \mathbb Q/\mathbb Z$, let $k=\gamma(s)$: then $q_K(k)=-d$, and the cyclic subgroup generated by $s+k$ is isotropic in $A_S\oplus A_K$, so it determines an overlattice of finite index $Q$ of $S\oplus K$ that is by construction a sublattice of $L$. Consider an isometry $\alpha\in O(S)$, denote $\overline\alpha(s)$ the induced isometry on $A_S$, and call $s'=\overline\alpha(s)$: then $q_S(s')=d$, and $\beta\coloneqq\gamma\circ\overline\alpha\circ\gamma^{-1}\in O(A_K)$, hence the subgroup generated by $s'+\beta(k)$ determines an overlattice of $S\oplus K$ isomorphic to $Q$ thanks to the previous theorem. \\
On the other hand, consider $Q$ such that $S\oplus K\hookrightarrow Q\hookrightarrow L$ and $Q/(S\oplus K)$ is cyclic: then $Q/(S\oplus K)$ is generated by an isotropic element in $A_{S\oplus K}=A_S\oplus A_K$, that is by construction of the form $s+k$, with $q_S(s)=d=-q_K(k)$.
\end{proof}

\begin{remark}\label{ino}
With $S,K,L$ as above, consider a primitive sublattice $H\subset K$, with $H=\langle h\rangle$, and suppose that $h/n\in K^*$ for some integer $n$: then, $L$ contains a cyclic overlattice $Q$ of $S\oplus K$, corresponding to the isotropic subgroup $\langle\gamma^{-1}(h/n)+h/n\rangle\subset A_{S\oplus K}$, as in Corollary \ref{corollario}. It contains also an overlattice $R$ of $S\oplus H$, $R\subseteq Q$, generated over $\mathbb Z$ by a $\mathbb Z$-basis of $S$ and the element $(s+h)/n$, where $s/n$ is a representative in $S^*$ of $\gamma^{-1}(h/n)\in A_S$.
Consider now an isometry $\psi\in O(K)$ such that $\tilde H=\psi(H)$ is isometric to $H$, $\tilde H=\langle\tilde h\rangle$: then $\tilde h/n=\psi(h)/n$ belongs to the same isometry class of $h/n$ in $A_K$, and $\langle\gamma^{-1}(\tilde h/n)+\tilde h/n\rangle$ is isotropic in $A_K$, so it defines an overlattice $\tilde R$ of $S\oplus\tilde H$ as above: the relation between $H$ and $\tilde R$ consists in the fact that $\tilde R$ is isomorphic to $R$ as a sublattice of $L$.
\end{remark}

\begin{theorem}[\cite{Nikulin1}, Prop. 1.14.1]\label{unicitaprimitivi}
For an even lattice $S$ of signature $(s_+,s_-)$ and discriminant form $q_S$, and an even unimodular indefinite lattice $L$ of signature $(l_+,l_-)$, all primitive embeddings of $S$ into $L$ are isomorphic if and only if the lattice $T$ with signature $(l_+-s_+,l_--s_-)$ and discriminant form $q_T=-q_S$ is unique and the homomorphism $O(T) \rightarrow O(q_T)$ is surjective.
\end{theorem}

\begin{remark}\label{nonesiste}
Using the notation of the theorem, if $\ell(S)>rk(L)-rk(S)$, no primitive embedding of $S$ in $L$ exists: indeed, if it existed, then $T$ would satisfy $rk(T)<\ell(T)$, which is impossible.
\end{remark}

Instead of proving directly that the conditions of Theorem \ref{unicitaprimitivi} hold, we can rely on some sufficient conditions that are based only on the invariants of $S$ and $L$:

\begin{proposition}[\cite{MM}, Cor. VIII.7.8]\label{condizione1}
Let $T$ be an indefinite lattice such that $rk(T)\geq 3$. Write $A_T=\mathbb Z/d_1\mathbb Z\oplus,\dots ,\oplus\mathbb Z/d_r\mathbb Z$
with $d_i\geq 1$ and $d_i\mid d_{i+1}$. Suppose that one of the following holds:
\begin{enumerate}
\item $d_1=d_2=2$;
\item $d_1=2, d_2=4$ and $d_3=_8 4$;
\item $d_1=d_2=4$.
\end{enumerate}
Then $T$ is uniquely determined by its signature and discriminant form, and the map $O(T)\rightarrow O(q_T)$ is surjective.
\end{proposition}

\begin{theorem}[Nikulin, see \cite{Morrison}, Thm. 2.8]\label{condizione2}
Let $S$ be an even lattice of signature $(s_+,s_-)$ and discriminant form $q_S$, and $L$ an even unimodular lattice of signature $(l_+,l_-)$. Suppose that $s_+<l_+,\ s_-<l_-,\ \lambda(A_S)\leq rk(L)-rk(S)-2$.
Then there exists a unique primitive embedding of $S$ into $L$.
\end{theorem}

\begin{corollary}[\cite{Nikulin1}, Rem. 1.14.5]\label{condizione3}
If $A_S\simeq(\mathbb Z/2\mathbb Z)^3\oplus A'$, the conditions of the previous theorem are satisfied.
\end{corollary}

\section{A symplectic automorphism $\tau$ of order 4 on the surface $X_4$}

\subsection{The surface $X_4$}\label{surface}

The surface $X_4$ is the unique K3 surface with transcendental lattice 
$T(X_4)=\begin{bmatrix}2 & 0 \\ 0 & 2 \end{bmatrix}$: it arises as desingularization of the quotient surface $A/\langle\sigma\rangle$, where $A$ is the abelian surface $E_i\times E_i$, $E_i$ is the elliptic curve of lattice $\langle 1,i\rangle$, and $\sigma$ is the automorphism of $A$ defined by $\sigma(e_1, e_2)=(ie_1, -ie_2)$; this surface is well known, see for instance \cite{ShiodaInose}, \cite{Vinberg}. \\
Define \textit{Jacobian fibration} a fibration $p:S\rightarrow\mathbb{P}^1$ whose generic fiber is a genus 1 curve, and that admits a global section $s$, denoted \textit{zero section}: the fiber over a generic point $F=p^{-1}(x)$ is an elliptic curve with the zero for the group law defined as $s(x)$. The \textit{Mordell-Weil group} $MW(p)$ of a Jacobian fibration is the group generated by all the sections, with the group law induced by that of the generic fiber: $MW(p)$ acts on $S$ by translation on each fiber, therefore if $S$ is a K3 surface it preserves its symplectic form.\\
Moreover, given a Jacobian fibration, the Mordell-Weil group  is linked to the Néron-Severi group of the surface by the following isomorphism (\cite{SchuttShioda}, Thm. 6.3):
\begin{equation}\label{eqn:MW}
MW(p)\simeq NS(S)/{\mathcal{T}(p)}
\end{equation}
where the \textit{trivial lattice} $\mathcal{T}(p)$ is the sublattice of $NS(S)$ generated by the generic fiber, the image of the zero section $s=s(\mathbb P^1)$ and the irreducible components of the reducible fibers which are not intersected by the curve $s$.

A description of  all the possible Jacobian fibrations on $X_4$ is provided by Nishiyama (\cite{Nishiyama}, Table 1.2): in particular there exists a fibration 
\[\pi: X_4\rightarrow\mathbb{P}^1 \quad \textrm{s.t.}\quad MW(\pi)\simeq\mathbb{Z}/{4\mathbb{Z}}\]
which provides a symplectic automorphism $\tau$ of order 4 on $X_4$ by means of a section $t_1$ that generates $MW(\pi)$.\\
The reducible fibers of $\pi$ are one of type $I_{4}$ and one of type $I_{16}$ (see \cite{Miranda}, Table IV.3.1). Call $B_0$ (respectively $C_0$) the component of $I_{4}$ (resp. $I_{16}$) intersected by $s$, and number the other components so that, for every $i\in\mathbb{Z}/4\mathbb{Z}, \ j\in\mathbb{Z}/16\mathbb{Z}$, $B_i$ intersects only $B_{(i+1)}$ and $B_{(i-1)}$, and $C_j$ intersects only $C_{(j+1)}$ and $C_{(j-1)}$.
Thus $\mathcal{T}(\pi)$ is generated by the class of the generic fiber $F$ of $\pi$, the curve $s$ and the components $B_i, C_j, \ i=1,2,3, \ j=1,\dots ,15$ of the reducible fibers: most of these curves are rational, so they have  self-intersection $-2$; the only exception is the class of the generic fiber $F$, which satisfies  $F^2=0$. The curves $B_1, B_2, B_3$ span the lattice $A_{3}$, and $C_1,\dots , C_{15}$ span the lattice $A_{15}$.\\
Using the height pairing formula (\cite{SchuttShioda}, \S 11.8), we can determine the components of the reducible fibers $B_i, C_j$ that have non-trivial intersection with a  non-zero section $t\in MW(\pi)$: in fact, since $t$ is a torsion section, its height is 0, so $i, j$ satisfy the  equation
\[0=h(t)=4-\big(\frac{i(4-i)}{4}+\frac{j(16-j)}{16}\big)\quad (height \ formula).\]
We will choose the following notation for the elements of $MW(\pi)$: the zero section $s$ intersects the components $B_0$ and $C_0$; the section $t_1$ intersects the components $B_2$ and $C_4$; the section $t_2$ intersects the components $B_0$ and $C_8$; the section $t_3$ intersects the components $B_2$ and $C_{12}$.
Notice that each of $t_1$ and $t_3$ generates $MW(\pi)$, whereas $t_2$ has order 2.\\
We can write $t_1,t_2,t_3$ in function of the basis of the trivial lattice $\mathcal{T}(\pi)$ using the information about their intersections:
\begin{align*}
t_1&=2F+s-\frac{B_1+2B_2+B_3}{2}-\frac{3C_1+6C_2+9C_3+\sum_{j=1}^{12} jC_{16-j}}{4}\\
t_2&=2F+s-\frac{\sum_{j=1}^7 j(C_j+C_{16-j})+8C_8}{2}\\
t_3&=2F+s-\frac{B_1+2B_2+B_3}{2}-\frac{\sum_{j=1}^{12} jC_{j}+9C_{13}+6C_{14}+3C_{15}}{4}.
\end{align*}
Since the discriminant group of $NS(X_4)$ is $(\mathbb Z/2\mathbb Z)^2$ (it is indeed the orthogonal complement to $T(X_4)$ in $\Lambda_{\mathrm{K3}}$), from \eqref{eqn:MW} and the equations above it can be readily seen that $NS(X_4)$ admits as a $\mathbb{Z}$-basis  $\mathcal{B}=\{F, s, t_1, B_1, B_2, B_3, C_1,\dots , C_{14}\}$.

\subsection{The action of $\tau^*$ on the second cohomology of $X_4$}

The symplectic automorphism $\tau$ induces an isometry $\tau^*$ on $NS(X_4)$ such that 
\begin{align*}
F\xmapsto{\tau^*}F \qquad  s\xmapsto{\tau^*}t_1\xmapsto{\tau^*}t_2\xmapsto{\tau^*}t_3\xmapsto{\tau^*}s\\
\tau^*(C_j)=C_{[j+4]_{16}}\qquad \tau^*(B_i)=B_{[i+2]_4}.
\end{align*}
where $[a]_{n}$ is the class of $a$ modulo $n$.
Therefore, we can easily identify two copies of $A_1$,  $\langle B_1\rangle$ and $\langle B_3\rangle$, exchanged by the action of $\tau^*$, and a set of four copies of $D_4$ on which $\tau^*$ acts as a cycle of order 4: $\{\langle s, C_{15}, C_0, C_1\rangle,\ \langle t_1, C_3, C_4, C_5\rangle,\ \langle t_2, C_7, C_8, C_9\rangle,$ $ \langle t_3, C_{11}, C_{12}, C_{13}\rangle\}$.
All these lattices are pairwise orthogonal, and the orthogonal complement in $NS(X)$ of the direct sum $D_4^{\oplus 4}\oplus A_1^{\oplus 2}$ is generated over $\mathbb Q$ by the vectors 
\begin{align*}
R_1=&-8F-4s+8t_1+4C_1+8C_2+13C_3+18C_4+15C_5+12C_6+10C_7+8C_8+ \\
&+6C_9+4C_{10}+3C_{11}+2C_{12}+C_{13}+3B_1+6B_2+3B_3, \\ 
R_2=&-4F-2s+2t_1+2C_1+4C_2+5C_3+6C_4+5C_5+4C_6+4C_7+4C_8+\\
&+4C_9+4C_{10}+3C_{11}+2C_{12}+C_{13}+B_1+2B_2+B_3, 
\end{align*}
whose intersection form satisfies $R_1^2=4,\ R_2^2=-4,\ R_1R_2=0$.
It can be also verified that $\tau^*R_1=R_1$, while $\tau^*R_2=-R_2$. Therefore, we have the following description:
\begin{proposition} Consider the sublattice
 $D_4^{\oplus 4}\oplus A_1^{\oplus 2}\oplus \langle 4 \rangle \oplus  \langle -4 \rangle$
of $NS(X_4)$ generated as above. The isometry $\tau^*$ acts on this sublattice as the cyclic permutation of order 4 on $D_4\oplus D_4\oplus D_4\oplus D_4$, as the cyclic permutation of order 2 on $A_1\oplus A_1$, as $id$ (the identity) on $ \langle4\rangle$ and as $-id$ on $ \langle-4\rangle$.
\end{proposition}

The lattice $D_4^{\oplus 4}\oplus A_1^{\oplus 2}\oplus \langle 4 \rangle \oplus  \langle -4 \rangle$ has discriminant group
\[(\mathbb{Z}/2\mathbb{Z}\times\mathbb{Z}/2\mathbb{Z})^4\times (\mathbb{Z}/2\mathbb{Z})^2\times\mathbb{Z}/4\mathbb{Z}\times\mathbb{Z}/4\mathbb{Z}\]
so it has index $2^6$ in $NS(X_4)$; the latter can be obtained by adding the following generators to the generators of $D_4^{\oplus 4}\oplus A_1^{\oplus 2}\oplus \langle 4 \rangle \oplus  \langle -4 \rangle$:
\begin{align}\label{generatoriNS}
R&\coloneqq(R_1+R_2)/2;\nonumber\\ 
a&\coloneqq R_1/4+R_2/4-(s+C_{15})/2-(t_1+C_5)/2=C_0+C_1+C_2+C_3+C_4;\nonumber\\ 
b&\coloneqq R_1/4-R_2/4-(t_1+C_3)/2-(t_2+C_9)/2=C_4+C_5+C_6+C_7+C_8;\nonumber\\ 
c&\coloneqq R_1/4+R_2/4-(t_2+C_7)/2-(t_3+C_{13})/2=C_8+C_9+C_{10}+C_{11}+C_{12};\\ 
d&\coloneqq R_1/4-R_2/4-(t_3+C_{11})/2-(s+C_1)/2=C_{12}+C_{13}+C_{14}+C_{15}+C_0;\nonumber\\ 
e&\coloneqq R_1/2-(C_3+C_5)/2-(C_{11}+C_{13})/2-B_1/2-B_3/2=t_1+t_3+C_4+C_{12}+B_2.\nonumber 
\end{align}

Now, $H^2(X_4,\mathbb{Z})$ is an overlattice of index 4 of the lattice $NS(X_4)\oplus T(X_4)$. 
Since $H^2(X_4,\mathbb{Z})$ is unimodular, following Theorem \ref{Nik1} we have to find a series of isotropic subgroups of $A_{NS(X_4)\oplus T(X_4)}$.\\
Denoting $\{\omega_1, \omega_2\}$ the $\mathbb{Z}$-basis of $T(X_4)$ for which the intersection matrix is $\begin{bmatrix}2 & 0 \\ 0 & 2 \end{bmatrix}$, the elements we have to add are:
\begin{gather}\label{generatoriH2}
f=(s+t_1+C_1+C_3+B_1+\omega_1)/2,\\
g=(s+t_1+C_1+C_3+B_3+\omega_2)/2.\nonumber
\end{gather}

\section{The action of $\tau^*$ and $(\tau^2)^*$ on the K3 lattice}\label{sec:azioni}

The main result of Nikulin's paper \cite{Nikulin2} is that there is a finite list of finite abelian groups which can act symplectically on K3 surfaces, and the action of each one of these groups on the K3 lattice is unique up to isometry (\cite{Nikulin2}, Thm. 4.7). This result enables us to deduce the action of any symplectic automorphism of order 4 (and of its square) on the second cohomology group of any K3 surface $S$ by looking at the surface $X_4$ with the action of the automorphism $\tau$ introduced in the previous section.

\subsection{A convenient description of the K3 lattice}\label{marking}

The isometry $\tau^*$ induced on $\Lambda_{\mathrm{K3}}$ by an automorphism $\tau$ of order 4 acts on the sublattice of finite index of $\Lambda_{\mathrm{K3}}\ W\coloneqq D_4^{\oplus 4}\oplus A_1^{\oplus 2}\oplus\langle -4\rangle\oplus\langle 4\rangle\oplus\langle 2\rangle\oplus\langle 2\rangle$ as the cycle $(1,2,3,4)$ on the four copies of $D_4$, as $(1,2)$ on the two copies of $A_1$, as $-id$ on ${\langle-4\rangle}$ and as $id$ on the remaining orthogonal components. 
The following diagram describes the situation: \\
\begin{equation}\label{eqn:sublattice}
\xymatrixcolsep{1pc}\xymatrix{
W\coloneqq &D_4 \ar@{}[r]|{\oplus} \ar@/^1pc/[r] &D_4 \ar@{}[r]|{\oplus}  \ar@/^1pc/[r] &D_4 \ar@{}[r]|{\oplus} \ar@/^1pc/[r] &D_4 \ar@{}[r]|{\oplus} \ar@/^1.2pc/[lll] &A_1 \ar@{}[r]|{\oplus}  \ar@/^1pc/[r]& A_1\ar@/^1pc/[l]\ar@{}[r]|(.45){\oplus} &{\langle-4\rangle} \ar@(ul,ur)^{-id}\ar@{}[r]|(.4){\oplus} &*[r]{\begin{bmatrix}
4 & 0 &  0 \\
0 & 2 &  0 \\
0 & 0 &  2
\end{bmatrix}} &\ \ar@(ur,dr)^{id} }\\
\end{equation}
Denote ${e_1,\dots , e_4,\ f_1,\dots , f_4,}$ ${g_1,\dots , g_4,\ h_1,\dots , h_4}$ the generators of the four copies of $D_4$, such that $e_3e_1=e_3e_2=e_3e_4=1$, and $\tau^*:e_i \mapsto f_i \mapsto g_i \mapsto h_i\mapsto e_i$ for $i=1,\dots , 4$; $a_1$ and $a_2$ the generators of the two copies of $A_1$;  $\sigma$ the generator of  $\langle-4\rangle$, $\rho$ the generator of  $\langle4\rangle$, $\omega_1$ and $\omega_2$ the generators of ${\begin{bmatrix}
 2 &  0 \\
0 &  2
\end{bmatrix}}$.
Then, the K3 lattice is obtained by adding to $W$ the elements (cf. \eqref{generatoriNS} and \eqref{generatoriH2})
\begin{align}\label{eqn:elements}
\chi&=(\rho+\sigma)/2; \nonumber\\
\alpha&=(\rho+\sigma)/4+(e_1+e_2+f_1+f_4)/2; \nonumber\\
\beta&=(\rho-\sigma)/4+(f_1+f_2+g_1+g_4)/2; \nonumber\\
\gamma&=(\rho+\sigma)/4+(g_1+g_2+h_1+h_4)/2; \nonumber\\
\delta&=(\rho-\sigma)/4+(h_1+h_2+e_1+e_4)/2; \\
\varepsilon&=(\rho+f_2+f_4+h_2+h_4+a_1+a_2)/2; \nonumber\\
\zeta&=(e_1+f_1+e_4+f_2+a_1+\omega_1)/2;\nonumber\\
\eta&=(e_1+f_1+e_4+f_2+a_2+\omega_2)/2 \nonumber;
\end{align}
the action of $\tau^*$ and $(\tau^2)^*$ on these elements is deduced by the one on the sublattice $W$ described above by $\mathbb Q$-linear extension: notice that ${\tau^*:\alpha \mapsto \beta \mapsto \gamma \mapsto \delta\mapsto \alpha}$, and that $(\tau^2)^*$ fixes $\varepsilon$ and $\chi$.

\subsection{Invariant and co-invariant lattices for the action of $\tau$ and $\tau^2$}\label{Omega4}

If $G$ is abelian and acts symplectically on a K3 surface, by uniqueness of the action of $G$ on $\Lambda_{\mathrm{K3}}$ we can define the \textit {invariant lattice} $\Lambda_{\mathrm{K3}}^G$ and the \textit{co-invariant lattice} $\Omega_G=(\Lambda_{\mathrm{K3}}^G)^{\perp_{\Lambda_{\mathrm{K3}}}}$. Moreover, the existence of a primitive embedding of $\Omega_G$ in the Néron-Severi lattice of a K3 surface is equivalent to the existence of a symplectic action of $G$ on that surface  (\cite{Nikulin2}, Thm. 4.15).
All the lattices $\Omega_G$ are known: see \cite{Morrison} and \cite{VGS} for $G=\mathbb Z/2\mathbb Z$, \cite{GS2} for $G=\mathbb Z/p\mathbb Z$, $p=3,5,7$, \cite{GS1} for the remaining cases; a list of all the invariant and co-invariant lattices for $G$ finite is provided in \cite{Hashimoto}.

From now on, denote $\Lambda_{\mathrm{K3}}^{\langle\tau\rangle}$ the invariant lattice, and $\Omega_4$ the co-invariant lattice for the action on $\Lambda_{\mathrm{K3}}$ induced by the automorphism of order 4 $\tau$. The lattice $\Omega_4$ has rank $14$ and discriminant $2^{10}$, as it was already computed by Nikulin (\cite{Nikulin2}, \S 10) by comparison with the lattice $M$ (see Section \ref{Y}).

The invariant elements for the action of $\tau^*$ on $W$ (see \eqref{eqn:sublattice}) are ${\kappa_i=e_i+f_i+g_i+h_i}$ for $i=1,\dots , 4,\ 
\kappa_5=a_1+a_2,\ \kappa_6=\rho,\ \kappa_7=\omega_1,\ \kappa_8=\omega_2$.
These elements span a sublattice of $\Lambda_{\mathrm{K3}}^{\langle\tau\rangle}$ of index $2^3$: to obtain $\Lambda_{\mathrm{K3}}^{\langle\tau\rangle}$, add the generators $(\kappa_2+\kappa_4)/2$ (that is, $\alpha+\beta+\gamma+\delta-\kappa_1-\kappa_6$), $(\kappa_5+\kappa_7+\kappa_8)/2$ and $(\kappa_1+\kappa_4+\kappa_6+\kappa_7+\kappa_8)/2$ (these are easily verified to be, in fact, integer elements in $\Lambda_{\mathrm{K3}}$).\\
Its orthogonal complement $\Omega_4$ is an overlattice of index $2^4$ of the lattice
 \[ \Delta=\left[ \begin{tabular}{ c | c | c | c }
  $D_4(2)$ & $D_4$ & $D_4$ & 0 \TBstrut\\
  \hline
  $D_4$ & $D_4(2)$ & $D_4$ & 0 \TBstrut\\
  \hline
  $D_4$ & $D_4$ & $D_4(2)$ & 0 \TBstrut\\
  \hline
  0 & 0 & 0 &$\begin{matrix}-4 & \enspace 0\\ \enspace 0&-4\end{matrix}$ \\
\end{tabular} \right]\]
spanned by the elements $e_i-f_i,\ e_i-g_i,\ e_i-h_i$ for $i=1,\dots , 4$, $a_1-a_2$ and $\sigma$: to obtain $\Omega_4$, we shall add to this lattice $\alpha-\beta,\ \alpha-\gamma,\ \alpha-\delta$ and one more class in $\frac{1}{2}\Delta$ integer in $\Lambda_{\mathrm{K3}}$, that we can choose to be $(e_2-g_2+e_4-g_4+a_1-a_2+\sigma)/2$.\\
The discriminant groups of $\Lambda_{\mathrm{K3}}^{\langle\tau\rangle}$ and ${\Omega_4}$ satisfy $A_{\Lambda_{\mathrm{K3}}^{\langle\tau\rangle}}\simeq(\mathbb{Z}/4\mathbb{Z})^4\times(\mathbb{Z}/2\mathbb{Z})^2\simeq A_{\Omega_4}.$

The sublattice of the K3 lattice $\Omega_2\coloneqq\Omega_{\mathbb{Z}/2\mathbb{Z}}$ co-invariant for an involution is isometric to the lattice $E_8(2)$ (see \cite{VGS} \S 1.3, \cite{Morrison} proof of Thm. 5.7). Considering the involution $\tau^2$, $\Omega_2$ is obviously contained in the co-invariant lattice for $\tau$, $\Omega_4$: in the basis of $\Omega_4$ described above, $\Omega_2$ is generated by the elements $\alpha-\gamma, \ \beta-\delta, \ e_1-g_1, \ e_3-g_3, \ f_1-h_1, \ f_2-h_2, \ f_3-h_3, \ f_4-h_4.$\\
Denote $R$ the orthogonal complement to $\Omega_2$ in $\Omega_4$: then $\Omega_4$ is an overlattice of $\Omega_2\oplus R$ such that $\Omega_4/(\Omega_2\oplus R)=(\mathbb Z/2\mathbb Z)^4$.

\section{Quotients}

For each of the abelian groups $G$ that can act symplectically on a K3 surface $X$, Nikulin provides (\cite{Nikulin2}, \S 5-7) a description of the singular locus of the quotient surface $X/G$, and of the \emph{exceptional lattice} $M_G$: this is the minimal primitive sublattice of $\Lambda_{\mathrm{K3}}$ containing all the exceptional curves of the minimal resolution $\widetilde{X/G}$ of $X/G$. Denoting $q: X\rightarrow X/G$ the quotient map, 
$H^2(\widetilde{X/G},\mathbb Z)$ is an overlattice of finite index of $q_*H^2(X,\mathbb Z)\oplus M_G$. 

\begin{remark}
The lattices $\Omega_G$ and $M_G$ are closely related via the quotient map: this fact allowed Nikulin to compute the rank and the discriminant group of $\Omega_G$ starting from the (simpler) exceptional lattice (see \cite{Nikulin2}, Lemma 10.2). Nikulin's results are correct only if $G$ is cyclic, and have been otherwise corrected by Garbagnati and Sarti in \cite{GS1}; for a complete account of the relation between $\Omega_G$ and $M_G$, see Whitcher's paper \cite{Whitcher}.
\end{remark}

Consider a K3 surface $X$ that admits a symplectic automorphism $\tau$ of order 4, and the (singular) quotient surfaces $Y=X/\tau, \ Z=X/\tau^2$; resolve the singularities of $Y$ and $Z$ to obtain the K3 surfaces $\tilde Y$, $\tilde Z$: then 
$\tau$ induces an involution $\hat\tau$ on $Z$ such that $Z/\hat\tau\simeq Y$, and this involution can be extended to $\tilde Z$, as we're going to show in the following sections. Denote the maps between these surfaces as in the following diagram:
\begin{equation}\label{diagrammaquozienti}
\xymatrix{
X \ar@{=}[rrrr] \ar[dd]_{q_4}\ar@{-->}[ddr]^{\pi_4} &\ &\ &\ &X \ar[d]^{q_2}\ar@{-->}[ld]_{\pi_2}\\
\ &\ &\ &\tilde Z\ar[d]^{\widehat{q_2}}\ar@{-->}[ld]_{\widehat{\pi_2}}\ar[r] &Z \ar[d]^{\overline{q}_2}\\
Y &\tilde Y \ar[l] \ar@{<->}[r]^{\simeq} &\widetilde{\tilde Z/\hat\tau} \ar[r] &\tilde Z/\hat\tau \ar[r] &Z/\hat\tau
}\end{equation}
\begin{remark}\label{YZisomorfi}
The surfaces $\tilde Y$ and $\widetilde{\tilde Z/\hat\tau}$ are isomorphic, because they are birationally equivalent K3 surfaces.
\end{remark}
We can describe the maps
\begin{align*}
\pi_{4*}&: \Lambda_{\mathrm{K3}}\simeq H^2(X, \mathbb Z)\xrightarrow{q_{4*}} q_{4*}H^2(X, \mathbb{Z}) \hookrightarrow H^2(\tilde Y, \mathbb{Z})\simeq \Lambda_{\mathrm{K3}} \\
\pi_{2*}&: \Lambda_{\mathrm{K3}}\simeq H^2(X, \mathbb Z)\xrightarrow{q_{2*}} q_{2*}H^2(X, \mathbb{Z}) \hookrightarrow H^2(\tilde Z, \mathbb{Z})\simeq \Lambda_{\mathrm{K3}}
\end{align*}
by defining them on the sublattice $W$ (see \eqref{eqn:sublattice}) in the first place, and subsequently on all of $\Lambda_{\mathrm{K3}}$ by $\mathbb{Q}$-linear extension to the elements presented in \eqref{eqn:elements}. 

The description of
\[\widehat{\pi_{2}}_*: H^2(\tilde Z, \mathbb Z)\rightarrow H^2(\tilde Y, \mathbb{Z})\]
will require some more effort: in fact, $H^2(\tilde Z, \mathbb Z)$ is an overlattice of finite index of $q_{2*}H^2(X,\mathbb Z)\oplus M_{\mathbb Z/2\mathbb Z}$, while $H^2(\tilde Y, \mathbb Z)$ is an overlattice of finite index of $q_{4*}H^2(X,\mathbb Z)\oplus M_{\mathbb Z/4\mathbb Z}=(\overline q_{2}\circ q_2)_*H^2(X,\mathbb Z)\oplus M_{\mathbb Z/4\mathbb Z}$; for now, notice that 
\begin{align}\label{picap2sottostella}
&\widehat{\pi_{2}}_*|_{M_{\mathbb Z/2\mathbb Z}}: M_{\mathbb Z/2\mathbb Z}\rightarrow M_{\mathbb Z/4\mathbb Z},\nonumber\\
&\widehat{\pi_{2}}_*|_{q_{2*}H^2(X,\mathbb Z)}: q_{2*}H^2(X,\mathbb Z)\rightarrow q_{4*}H^2(X,\mathbb Z).
\end{align}

\subsection{The image of $H^2(X,\mathbb Z)$ via the maps $\pi_{4*}$ and $\pi_{2*}$}\label{sec:maps}

\begin{proposition}\label{prop:sottostella}
The maps $\pi_{2*},\widehat{\pi_{2}}_*$ and $\pi_{4*}=\widehat{\pi_{2}}_*\circ \pi_{2*}$ act in the following way on $W$ and its image in $\pi_{2*}H^2(X,\mathbb Z)$:

\xymatrix@C=0.001pc{
{\begin{matrix}D_4 \\ e_1\dots e_4\end{matrix}} \ar[drr]_{\pi_{2*}}&{\begin{matrix}\oplus\\ \ \end{matrix}} &{\begin{matrix}D_4 \\ f_1\dots f_4\end{matrix}} \ar[drr] &{\begin{matrix}\oplus\\ \ \end{matrix}} &{\begin{matrix}D_4 \\ g_1\dots  g_4\end{matrix}} \ar[dll] &{\begin{matrix}\oplus\\ \ \end{matrix}} &{\begin{matrix}D_4 \\ h_1\dots h_4\end{matrix}} \ar[dll]^{\pi_{2*}} \ar@{}[r]|{\begin{matrix}\oplus\\ \ \end{matrix}} &{\begin{matrix}A_1 \\ a_1\end{matrix}} \ar[d]^{\pi_{2*}} &{\begin{matrix}\oplus\\ \ \end{matrix}} &{\begin{matrix}A_1 \\ a_2\end{matrix}}\ar[d]_{\pi_{2*}} \ar@{}[r]|(.45){\begin{matrix}\oplus\\ \ \end{matrix}} &{\begin{matrix}\langle-4\rangle\\ \sigma\end{matrix}}\ar[d]^{\pi_{2*}} \ar@{}[r]|(.35){\begin{matrix}\oplus\\ \ \end{matrix}} &{\begin{matrix}\begin{bmatrix}
4 & 0 &  0 \\
0 & 2 &  0 \\
0 & 0 &  2\end{bmatrix}\\ \rho, \omega_1,\omega_2\end{matrix}}\ar[d]^{\pi_{2*}}\\ 
\ &\ &{\begin{matrix}D_4 \\ \hat e_1\dots \hat e_4\end{matrix}} \ar[dr]_{\widehat{\pi_{2}}_*} &{\begin{matrix}\oplus\\ \ \end{matrix}} &{\begin{matrix}D_4 \\ \hat f_1\dots \hat f_4\end{matrix}} \ar[dl]^{\widehat{\pi_{2}}_*} \ar@{}[rrr]|{\begin{matrix}\oplus\\ \ \end{matrix}} &\ &\  &{\begin{matrix}A_1(2) \\ \hat a_1\end{matrix}}\ar[dr]_{\widehat{\pi_{2}}_*} &{\begin{matrix}\oplus\\ \ \end{matrix}} & {\begin{matrix}A_1(2) \\ \hat a_2\end{matrix}}\ar[dl]^{\widehat{\pi_{2}}_*}\ar@{}[r]|(.53){\begin{matrix}\oplus\\ \ \end{matrix}} &{\begin{matrix}\langle-8\rangle\\ \hat\sigma\end{matrix}}\ar@{}[r]|(.35){\begin{matrix}\oplus\\ \ \end{matrix}} &{\begin{matrix}\begin{bmatrix}
8 & 0 &  0 \\
0 & 4 &  0 \\
0 & 0 &  4\end{bmatrix}\\ \hat\rho, \hat\omega_1,\hat\omega_2\end{matrix}}\ar[d]^{\widehat{\pi_{2}}_*}\\ 
\ &\ &\ &{\begin{matrix}D_4 \\ \overline e_1\dots \overline e_4\end{matrix}} \ar@{}[rrrrr]|(.527){\begin{matrix}\oplus\\ \ \end{matrix}} &\ &\ &\ &\ &{\begin{matrix}A_1(2) \\ \overline a\end{matrix}} \ar@{}[rrr]|(.45){\begin{matrix}\oplus\\ \ \end{matrix}} &\ &\ &{\begin{matrix}\begin{bmatrix}
16 & 0 &  0 \\
0 & 8 &  0 \\
0 & 0 &  8\end{bmatrix}\\ \overline\rho, \overline\omega_1,\overline\omega_2\end{matrix}}\\
}
\end{proposition}

\begin{proof}
Since $\pi_{4}$ is a finite morphism of degree 4, we can compute the intersection form of $\pi_{4*}W$ via the push-pull formula: 
\[\pi_{4*}x\cdot\pi_{4*}y=\frac{1}{4}(\pi^{*}_4\pi_{4*}x\cdot\pi^{*}_4\pi_{4*}y)\]
where $\pi^{*}_4\pi_{4*}x=\sum_{k=0}^3(\tau^k)^*(x)$. 
The behaviour of $\pi_{2*}$ and $\widehat{\pi_{2}}_*$ can be similarly determined.
\end{proof}
\begin{corollary}
The embedding $q_{4*}H^2(X,\mathbb Z)\hookrightarrow H^2(\tilde Y,\mathbb Z)$ is unique up to isometries of the latter, and the same holds for $q_{2*}H^2(X,\mathbb Z)\hookrightarrow H^2(\tilde Z,\mathbb Z)$.
\end{corollary}
\begin{proof}
Computing $q_{4*}H^2(X,\mathbb Z)$ and $q_{2*}H^2(X,\mathbb Z)$ by $\mathbb Q$-linear extension of the maps in the proposition above to the elements in \eqref{eqn:elements}, it can be seen that the lattice $q_{4*}H^2(X,\mathbb Z)$  is even, indefinite, and it has $rk=8$ and $\ell=4$, so it satisfies the conditions of (\cite{Nikulin1}, Thm. 1.14.2); similarly, $q_{2*}H^2(X,\mathbb Z)$ is even, indefinite, it has $rk=14$ and $\ell=6$, and the same result can be applied.
\end{proof}

\subsection{The resolution of $Z=X/\tau^2$, and the lattice $\Gamma$}\label{sec:resolutionZ}

For any symplectic involution $\iota$ on a K3 surface $X$, the quotient surface $X/\iota$ has 8 isolated singularities, that are ordinary double points: to resolve it, it is sufficient to blow up these points once. Therefore, the exceptional lattice $N\coloneqq M_{\mathbb Z/2\mathbb Z}$ for the quotient $X/\iota$ (usually called \emph{Nikulin lattice}) satisfies ${N\otimes\mathbb Q=A_1^{\oplus 8}\otimes\mathbb Q}$: more precisely, if $\{n_1,\dots , n_8\}$ is a $\mathbb Z$-basis of $A_1^{\oplus 8}$, then a set of generators over $\mathbb Z$ for $N$ can be obtained by adding to this list the element ${\nu=(n_1+\dots +n_8)/2}$.

The second integral cohomology of the K3 surface $\tilde Z$, the minimal resolution of the quotient $Z=X/\tau^2$, can be described as an overlattice of index $2^6$ of $\pi_{2*}H^2(X,\mathbb Z)\oplus N$: this is done via an isomorphism of the discriminant groups of $\pi_{2*}H^2(X,\mathbb Z)$ and $N$, as described in Theorem \ref{thm:overlattices}. The generators that we need to add are the following:
\begin{align}\label{z}
z_1&=(\hat\beta+\hat f_1+\hat f_2+\hat a_2+\hat\eta)/2+(n_2+n_8)/2 \nonumber\\
z_2&=\hat\varepsilon/2+(n_3+n_8)/2\nonumber\\
z_3&=\hat a_2/2+(n_4+n_8)/2\\
z_4&=(\hat\varepsilon+\hat a_2+\hat\rho)/2+(n_5+n_8)/2\nonumber\\
z_5&=(\hat\beta+\hat f_1+\hat f_2+\hat\rho+\hat\chi+\hat\eta)/2+(n_6+n_8)/2\nonumber\\
z_6&=(\hat\beta+\hat f_1+\hat f_2+\hat\varepsilon+\hat a_2+\hat\rho+\hat\zeta)/2+(n_7+n_8)/2,\nonumber
\end{align}
where $\hat\beta=\pi_{2*}\beta \ (=\pi_{2*}\delta)$, and similarly the cap over the other elements of $\Lambda_{\mathrm{K3}}$ denotes their image via $\pi_{2*}$. These generators were already known in the general case of a symplectic involution (see \cite{VGS}, Lemma 1.10).

\begin{definition}\label{def:Gamma}
Notice that, while $\Omega_2\subset\Omega_4$ is annichilated by $\pi_{2*}$, the same is not true for its orthogonal complement in $\Omega_4$, the lattice $R$ (see Section \ref{Omega4}). Define $\hat R\coloneqq\pi_{2*}\Omega_4$: it is an overlattice of $\pi_{2*}R$, and it is spanned by ${\hat e_1-\hat f_1},\ {\hat\alpha-\hat\beta},\ {\hat e_3-\hat f_3},\ {\hat e_4-\hat f_4},\ {\hat\sigma},$ ${(\hat a_1-\hat a_2+\hat\sigma)/2}$. It has rank 6 and discriminant group $(\mathbb Z/2\mathbb Z)^4\times(\mathbb Z/4\mathbb Z)^2$.\\
Define $\Gamma$ the lattice of rank 14 and discriminant group $(\mathbb Z/2\mathbb Z)^6\times(\mathbb Z/4\mathbb Z)^2$, obtained as an overlattice of $\hat R\oplus N$ by adding to the list of generators the elements
\begin{equation}\label{incNR}
x_1=\frac{n_3+n_4+n_5+n_8+\hat\sigma}{2}, \qquad x_2=\frac{n_3+n_4+(\hat e_1-\hat f_1)+(\hat e_4-\hat f_4)}{2}+\frac{(\hat a_1-\hat a_2+\hat\sigma)}{4}.\end{equation}
\end{definition}

The lattice $\Gamma$ can be primitively embedded in $H^2(\tilde Z,\mathbb Z)$ with the lattice $S=\langle\hat e_1+\hat f_1,\ \hat\alpha+\hat\beta,\ \hat e_3+\hat f_3,\ {\hat e_2+\hat f_2},$ $\hat\rho,\ {(\hat a_1+\hat a_2+\hat\rho)/2}, (\hat\rho+\hat\omega_1+\hat\omega_2)/2, \hat\omega_2\rangle$ as orthogonal complement; indeed, we can obtain $H^2(\tilde Z,\mathbb Z)$ as overlattice of finite index of $S\oplus\Gamma$ by adding the generators:
\begin{align}\label{z2}
z'_1&= (\hat\alpha+\hat\beta)/2+ (\hat\alpha-\hat\beta)/2,\nonumber\\
z'_2&=(\hat e_3+\hat f_3)/2+(\hat e_3-\hat f_3)/2,\nonumber\\
z'_3&=(\hat e_2+\hat f_2)/2+(\hat e_4-\hat f_4)/2+(n_3+n_4+n_5+n_8)/2,\nonumber\\
z'_4&=(\hat a_1+\hat a_2+\hat\rho)/4+(n_2+n_3+n_4+n_6)/2+(\hat e_1-\hat f_1+\hat e_4-\hat f_4)/2,\\
z'_5&=(\hat\rho+\hat\omega_1+\hat\omega_2)/4+(n_2+n_7)/2+(\hat e_1-\hat f_1+\hat e_4-\hat f_4)/2,\nonumber\\
z'_6&=\hat\omega_2/2+(n_2+n_6)/2\nonumber,\\
z'_7&=(\hat e_1+\hat f_1+\hat e_2+\hat f_2+(\hat a_1+\hat a_2+\hat\rho)/2+\hat\omega_2)/4+\nonumber\\
       &\quad +(\hat\alpha-\hat\beta)/2+(3n_5+2n_6+n_8)/4+(x_1+x_2)/2,\nonumber\\
z'_8&=\hat\rho/4+(2n_2+n_3+3n_4+n_5+2n_6+3n_8)/4+x_1/2.\nonumber
\end{align}

\begin{remark}
The primitive embedding $\Gamma\hookrightarrow\Lambda_{\mathrm{K3}}$ is unique up to isometries of $\Lambda_{\mathrm{K3}}$, because Theorem \ref{unicitaprimitivi} holds: in fact the orthogonal complement $S$ of $\Gamma$ satisfies the first condition of Proposition \ref{condizione1}.
\end{remark}

\subsection{The map $\widehat{\pi_{2}}_*$ and the resolution of $Y=X/\tau$}\label{Y}

The action of a symplectic automorphism of order 4 $\tau$ on a K3 surface $X$ has always exactly eight isolated points on $X$ with non trivial stabilizer: four of them are fixed by $\tau$, and four more exchanged by $\tau$ (so they are fixed by $\tau^2$) (see \cite{Nikulin2}, \S 5, case 2); therefore the singular locus of the quotient $X/\tau$ consists of six isolated points, two of which are resolved by blowing up once (thus introducing two rational curves in the quotient surface), and each of the other four by three curves in $A_3$ configuration. \\
The exceptional lattice $M\coloneqq M_{\mathbb Z/4\mathbb Z}$ therefore satisfies $M\otimes\mathbb Q=( A_3^{\oplus 4}\oplus A_1^{\oplus 2})\otimes\mathbb Q$; calling $\tilde m^1, \ \tilde m^2$ the generators of the two copies of $A_1$, and $m^i_1,\ m^i_2,\ m^i_3$ the generators of the $i$-th copy of $A_3$, such that $m^i_2$ is the curve that intersects both $m^i_1$ and $m^i_3$, then a set of $\mathbb Z$-generators for $M$ consists of all these elements, and (see \cite{Nikulin2}, \S 6, Def. 6.2, case 1b) the class
\begin{equation}\label{mu}
\mu=\frac{\sum_{i=1}^4(m^i_1+2m^i_2+3m^i_3)}{4}+\frac{\tilde m^1+\tilde m^2}{2}.
\end{equation}
We can now describe the second integral cohomology of the K3 surface $\tilde Y$, the minimal resolution of the quotient $Y=X/\tau$: the discriminant group of each of the orthogonal summands $M$ and  $\pi_{4*}H^2(X,\mathbb Z)$ is isomorphic to $(\mathbb{Z}/4\mathbb{Z})^2\times(\mathbb{Z}/2\mathbb{Z})^2$, and the elements that define $\Lambda_{\mathrm{K3}}$
as an overlattice of $M \oplus \pi_{4*}H^2(X,\mathbb Z)$ are:
\begin{align}\label{y}
y_1&=\big(m^2_1+2m^2_2+3m^2_3+m^3_1+2m^3_2+3m^3_3\big)/4+\tilde m^2/2 + (\overline e_1+\overline e_4+\overline\zeta+\overline\eta)/2+(\overline a+\overline\chi)/4 \nonumber\\ 
y_2&=(m^3_1+2m^3_2+3m^3_3+m^4_1+2m^4_2+3m^4_3)/4+\tilde m^2/2 + (\overline a+\overline\zeta+3\overline\eta)/4 \nonumber\\
y_3&=(m^4_1+m^4_3+\tilde m^2)/2 + (\overline e_1+\overline e_4+\overline\alpha+ \overline\chi+\overline\zeta)/2\\
y_4&=(\tilde m^1+\tilde m^2)/2 + \overline a/2, \nonumber
\end{align}
where $\overline e_i=\pi_{4*} e_i=\pi_{4*} f_i=\pi_{4*} g_i=\pi_{4*} h_i,\ \overline a=\pi_{4*} a_1=\pi_{4*} a_2$, and similarly $\overline{\star}=\pi_{4*}(\star)$.

The exceptional lattice $M$ can be also computed from the image of $N$ via $\widehat{\pi_2}_*$ (see Rmk. \ref{picap2sottostella}) with the resolution of the singularities that arise from the quotient: in fact, the involution $\hat\tau$ (that is induced on $\tilde Z$ by the action of $\tau$ on $X$) acts by fixing two points on each of the the four exceptional curves of $\tilde Z$ corresponding to the four points of $X$ fixed by $\tau$, and by exchanging pairwise the remaining four exceptional curves (these correspond to the four points fixed only by $\tau^2$). Therefore, the invariant lattice for the action of $\hat\tau^*$ on $N$ is the sublattice spanned by the four invariant curves, and the sum of the pairs of exchanged curves.\\
In the lattice $\Gamma$ (see Def. \ref{def:Gamma}) the orthogonal complement of the invariant lattice for the action of $\hat\tau^*$ is a copy of $\Omega_2$: this is indicative of the fact that the surface $\tilde Z$ admits a symplectic involution $\hat\tau$.

\begin{remark}\label{Omega2suZ}
The curves of $N$ were numbered in such a way that the gluing between $N$ and $\pi_{2*}H^2(X,\mathbb Z)$ is described by the elements in \eqref{z}: since the action of $\hat\tau^*$ on $\pi_{2*}H^2(X,\mathbb Z)$ determines the action of $\hat\tau^*$ on $N$ via this gluing, we find accordingly that $\hat\tau^*$ fixes $n_1, n_2, n_6, n_7$ and exchanges $n_3$ with $n_4$, $n_5$ with $n_8$.
\end{remark}

\begin{proposition}\label{immagineGamma}
It holds $\widehat{\pi_{2}}_*(\Gamma)\subset M$: more precisely, $\widehat{\pi_{2}}_*$ annichilates $\hat R$, because $\widehat{\pi_{2}}_*\hat R=\widehat{\pi_{2}}_*\pi_{2*}\Omega_4=\pi_{4*}\Omega_4=0$, and is defined on the $\mathbb Q$-generators of $N$ as follows:
\begin{align*}
\overline n_1\coloneqq\widehat{\pi_{2}}_*n_1&=m^1_1+2m^1_2+m^1_3,\quad 
\overline n_2\coloneqq\widehat{\pi_{2}}_*n_2=m^3_1+2m^3_2+m^3_3,\\
\overline n_6\coloneqq\widehat{\pi_{2}}_*n_6&=m^2_1+2m^2_2+m^2_3,\quad
\overline n_7\coloneqq\widehat{\pi_{2}}_*n_7=m^4_1+2m^4_2+m^4_3,\\
\overline n_3\coloneqq\widehat{\pi_{2}}_*n_3&=\widehat{\pi_{2}}_*n_4=\tilde m^1,\qquad
\overline n_5\coloneqq\widehat{\pi_{2}}_*n_5=\widehat{\pi_{2}}_*n_8=\tilde m^2,
\end{align*}
so that $\widehat{\pi_{2}}_*A_1^{\oplus 8}=A_1(2)^{\oplus 4}\oplus A_1^{\oplus 2}$.
\end{proposition}
\begin{proof}
Let $k=1,2,6,7,\ j=1,\dots,8$. The surface $\tilde Z/\hat\tau$ is singular in eight points, two on each of the curves $\hat q_{2*}n_k$; consider the blow-up $\beta:\tilde Y\rightarrow \tilde Z/\hat\tau$ of the singular points: then, the curve $\overline n_j\coloneqq\widehat{\pi_{2}}_*n_j$ is the strict transform of $\hat q_{2*}n_j$. By push-pull we have
$\overline n_k^2=-4$, $\overline n_3^2=-2=\overline n_5^2$.
Consider the following diagram:
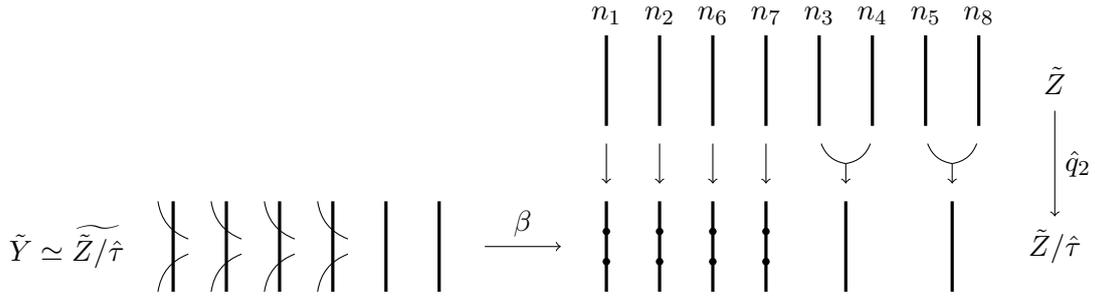
\begin{center}
\begin{tikzpicture}
\draw[very thick] (-5.4,0) -- (-5.4,1.2) node[above]{$n_1$};
	\node at (-5.4,-0.1)(a){};
	\node at (-5.4,-0.9)(a1){};
     \draw[->] (a) to (a1);
\draw[very thick] (-5.4,-1) -- (-5.4,-2.2); 
\filldraw[black] (-5.4,-1.4) circle (1.2pt);
\filldraw[black] (-5.4,-1.8) circle (1.2pt);

\draw[very thick] (-4.7,0) -- (-4.7,1.2) node[above]{$n_2$};
	\node at (-4.7,-0.1)(b){};
	\node at (-4.7,-0.9)(b1){};
     \draw[->] (b) to (b1);
\draw[very thick] (-4.7,-1) -- (-4.7,-2.2); 
\filldraw[black] (-4.7,-1.4) circle (1.2pt);
\filldraw[black] (-4.7,-1.8) circle (1.2pt);

\draw[very thick] (-4,0) -- (-4,1.2) node[above]{$n_6$};
	\node at (-4,-0.1)(c){};
	\node at (-4,-0.9)(c1){};
     \draw[->] (c) to (c1);
\draw[very thick] (-4,-1) -- (-4,-2.2); 
\filldraw[black] (-4,-1.4) circle (1.2pt);
\filldraw[black] (-4,-1.8) circle (1.2pt);

\draw[very thick] (-3.3,0) -- (-3.3,1.2) node[above]{$n_7$};
	\node at (-3.3,-0.1)(d){};
	\node at (-3.3,-0.9)(d1){};
     \draw[->] (d) to (d1);
\draw[very thick] (-3.3,-1) -- (-3.3,-2.2); 
\filldraw[black] (-3.3,-1.4) circle (1.2pt);
\filldraw[black] (-3.3,-1.8) circle (1.2pt);

\draw[very thick] (-2.6,0) -- (-2.6,1.2) node[above]{$n_3$};
\draw[very thick] (-1.9,0) -- (-1.9,1.2) node[above]{$n_4$};
	\node at (-2.6,-0.1)(s){};
     \draw (s) to[bend right] (-2.25,-0.5);
	\node at (-1.9,-0.1)(t){};
     \draw (t) to[bend left] (-2.25,-0.5);
	\node at  (-2.25,-0.9)(st){};
	\draw[->]  (-2.25,-0.5) -- (st);
\draw[very thick] (-2.25,-1) -- (-2.25,-2.2) ; 

\draw[very thick] (-1.2,0) -- (-1.2,1.2) node[above]{$n_5$};
\draw[very thick] (-0.5,0) -- (-0.5,1.2) node[above]{$n_8$};
	\node at (-1.2,-0.1)(u){};
     \draw (u) to[bend right] (-0.85,-0.5);
	\node at (-0.5,-0.1)(v){};
     \draw (v) to[bend left] (-0.85,-0.5);
	\node at  (-0.85,-0.9)(uv){};
	\draw[->]  (-0.85,-0.5) -- (uv);
\draw[very thick] (-0.85,-1) -- (-0.85,-2.2); 

\node at (0.5,0.6){$\tilde Z$};
\draw[->]  (0.5,0.2) --  (0.5,-1.2) node[midway,right]{$\hat q_2$};
\node at (0.5,-1.6){$\tilde Z/\hat\tau$};
\node at (-12.5,-1.6){$\tilde Y\simeq\widetilde{\tilde Z/\hat\tau}$};

\draw[->](-7,-1.6) -- (-6,-1.6) node[midway,above]{$\beta$};

\draw[very thick](-7.6,-1) --(-7.6,-2.2); 
\draw[very thick](-8.3,-1) --(-8.3,-2.2); 

\draw[very thick](-9,-1) --(-9,-2.2);
\draw(-9.2,-1) to[bend right](-8.8,-1.5);
\draw(-9.2,-2.2) to[bend left](-8.8,-1.7);

\draw[very thick](-9.7,-1) --(-9.7,-2.2);
\draw(-9.9,-1) to[bend right](-9.5,-1.5);
\draw(-9.9,-2.2) to[bend left](-9.5,-1.7);

\draw[very thick](-10.4,-1) --(-10.4,-2.2);
\draw(-10.6,-1) to[bend right](-10.2,-1.5);
\draw(-10.6,-2.2) to[bend left](-10.2,-1.7);

\draw[very thick](-11.1,-1) --(-11.1,-2.2);
\draw(-11.3,-1) to[bend right](-10.9,-1.5);
\draw(-11.3,-2.2) to[bend left](-10.9,-1.7);
\end{tikzpicture}
\end{center}
Firstly, notice that either $(\overline n_3, \overline n_5)=(\tilde m^1,\tilde m^2)$ or $(\overline n_3, \overline n_5)=(\tilde m^2,\tilde m^1)$.\\
The eight exceptional curves introduced with the blow-up $\beta$ span a new copy of the Nikulin lattice 
\[N_Y=\langle m^i_1, \ m^i_3\rangle_{i=1}^4\subset M;\] the pullback $\beta^*\hat q_{2*}n_k$ of each singular curve  is an $A_3$ lattice, and the class $\overline n_k$ (being the strict transform) is by definition orthogonal to the exceptional curves: thus we find that $\overline n_k=m^i_1+2m^i_2+m^i_3$ for some $i$.
To determine which copy of $A_3,\ A_1$ in $M$ each $\overline n_j$ corresponds to, we still have to require that the image of the elements $z_i$ defined in \eqref{z} be integer in $H^2(\tilde Y, \mathbb Z)$; this forces the definition of  $\widehat{\pi_{2}}_*$ as stated.
\end{proof}
\begin{corollary}\label{MinquozienteZ}
The lattice $M$ is an overlattice of index $2^5$ of the lattice $\widehat{\pi_{2}}_*N\oplus N_Y$, obtained by adding as generators the elements $m^1_2=(\overline n_1-m^1_1-m^1_3)/2$, $m^2_2=(\overline n_6-m^2_1-m^2_3)/2$, $m^3_2=(\overline n_2-m^3_1-m^3_3)/2$, $m^4_2=(\overline n_7-m^4_1-m^4_3)/2$ and $\mu$ (see \eqref{mu}).
\end{corollary}
\begin{proof}
The Nikulin lattice $N$ is an overlattice of index 2 of $A_1^{\oplus 8}=\langle a_1,\dots , a_8\rangle$ obtained by adding as generator the element $\nu=(a_1+,\dots , + a_8)/2$: in our case
\[\nu_Y=\frac{\sum_{i=1}^4(m^i_1+m^i_3)}{2};\]
the element $\mu$ defined in \eqref{mu} can then be written as
\[\mu=\frac{\nu_Y+m^1_2+m^2_2+m^3_2+m^4_2+m^1_3+m^2_3+m^3_3+m^4_3+\tilde m^1+\tilde m^2}{2},\]
so it generates an overlattice (of index 2) of the overlattice (of index $2^4$) of $\widehat{\pi_{2}}_*N\oplus N_Y$ generated by the $m^i_2, \ i=1,\dots , 4$.
\end{proof}

Notice that we have the following equalities:
\begin{align*}H^2(\tilde Y, \mathbb Q)&=(\pi_{4*}H^2(X, \mathbb Z)\oplus M)\otimes\mathbb Q=(\pi_{4*}H^2(X, \mathbb Z)\oplus\widehat{\pi_{2}}_*N\oplus N_Y)\otimes\mathbb Q=\\ 
&=(\widehat{\pi_{2}}_*(\pi_{2*}H^2(X, \mathbb Z)\oplus N)\oplus N_Y)\otimes\mathbb Q=(\widehat{\pi_{2}}_*H^2(\tilde Z, \mathbb Z)\oplus N_Y)\otimes\mathbb Q.
\end{align*}
Working on $\mathbb Z$, we recover the first equality using the $y_i$'s in \eqref{y}, and the second one as in Proposition \ref{MinquozienteZ}; the next one is trivial, and for the last one we use the $z_i$'s in \eqref{z}.
Thus, the lattice $H^2(\tilde Y, \mathbb Z)$ can be also described directly as an overlattice of finite index of $\widehat{\pi_{2}}_*H^2(\tilde Z, \mathbb Z)\oplus N_Y$, allowing for an easier computation of the map $\widehat{\pi_{2}}^*$ in Section \ref{dualmaps}: to do this, use for $\widehat{\pi_{2}}_*H^2(\tilde Z, \mathbb Z)$ the $\mathbb Z$-basis $\{\overline e_1, \overline\alpha, \overline e_3,\overline e_4, \overline a, \overline\chi, \overline\zeta, \overline\eta,\widehat{\pi_{2}}_*n_6, \widehat{\pi_{2}}_*\nu, \widehat{\pi_{2}}_*z_1,\widehat{\pi_{2}}_*z_2,$ $\widehat{\pi_{2}}_*z_5, \widehat{\pi_{2}}_*z_6\}$, with the $z_i$'s defined in \eqref{z}, and for $N_Y$ the $\mathbb Z$-basis $\{\nu_Y, m^1_3,m^2_1,m^2_3,m^3_1,$ $m^3_3,m^4_1,m^4_3\}$; then, the gluing elements are
\begin{align}\label{y'}
y'_1&=(\overline a+\overline\chi+\widehat{\pi_{2}}_*\nu+m^1_3+m^2_1+m^3_1+m^4_3)/2\nonumber\\
y'_2&=(\overline\alpha+\overline\zeta+\widehat{\pi_{2}}_* z_1+\widehat{\pi_{2}}_*z_5+m^2_3+m^3_3)/2\nonumber\\
y'_3&=(\overline\chi+\overline\zeta+\overline\eta+\widehat{\pi_{2}}_* n_6+m^4_1+m^4_3)/2\\
y'_4&=(\overline e_4+\overline a+\overline\eta+\widehat{\pi_{2}}_* z_1+\widehat{\pi_{2}}_*z_6+m^3_3+m^4_3)/2\nonumber\\
y'_5&=(\overline\alpha+\overline a+\overline\eta+\widehat{\pi_{2}}_*z_1+\widehat{\pi_{2}}_* z_5+m^2_3+m^3_1+m^4_1+m^4_3)/2\nonumber\\
y'_6&=(\widehat{\pi_{2}}_* n_6+m^2_1+m^2_3)/2.\nonumber
\end{align}

\subsection{The dual maps}\label{dualmaps}

We're now going to define the dual maps
\begin{align*}
\pi_{4}^*&: H^2(\tilde Y, \mathbb Z)\rightarrow H^2(X, \mathbb{Z}) \\
\widehat{\pi_{2}}^*&: H^2(\tilde Y, \mathbb Z)\rightarrow H^2(\tilde Z, \mathbb{Z}) \\
\pi_{2}^*&: H^2(\tilde Z, \mathbb Z)\rightarrow H^2(X, \mathbb{Z})
\end{align*}
using the descriptions of $H^2(\tilde Y, \mathbb Z)$ as an overlattice respectively of $\pi_{4*}\Lambda_{\mathrm{K3}}\oplus M$ and $\widehat{\pi_{2*}}\Lambda_{\mathrm{K3}}\oplus N_Y$, and of $H^2(\tilde Z, \mathbb Z)$ as an overlattice of $\pi_{2*}\Lambda_{\mathrm{K3}}\oplus N$.

\begin{proposition}
\begin{enumerate}
\item  The map $\pi_{4}^*$ annichilates $M$, and acts on  $\pi_{4*}W\subset \pi_{4*}\Lambda_{\mathrm{K3}}$ as\\
\[\xymatrix@R=0.01pc @C=0.3pc{
{\pi_{4}}^*: &D_4 \ar@{}[r]|(.4){\oplus} &A_1(2)  \ar@{}[r]|(.4){\oplus} &{\begin{bmatrix}
16 & 0 &  0 \\
0 & 8 &  0 \\
0 & 0 &  8\end{bmatrix}} &\ \ar[rr] 
&\ &{\quad D_4^{\oplus{4}}}  \ar@{}[r]|{\oplus} &A_1^{\oplus 2}  \ar@{}[r]|(.35){\oplus} 
&{\begin{bmatrix}
4 & 0 &  0 \\
0 & 2 &  0 \\
0 & 0 &  2\end{bmatrix}}\\
\ &{\begin{matrix}\overline e_1\\ \overline e_2\\ \overline e_3\\ \overline e_4\end{matrix}} &\overline a  &{\overline\rho,\ \overline\omega_1, \ \overline\omega_2} &\ \ar@{|->}[rr] 
&\ &{\begin{matrix}e_1+f_1+g_1+h_1\\ e_2+f_2+g_2+h_2\\ e_3+f_3+g_3+h_3\\ e_4+f_4+g_4+h_4\end{matrix}}&{2a_1+2a_2} 
&{4\rho,4\omega_1,4\omega_2}
}\]
Its action can be extended to $\pi_{4*}\Lambda_{\mathrm{K3}}$ adding these elements (and their respective images to the image lattice):
$\overline\alpha =\overline\rho/4+\overline e_1+ (\overline e_2+\overline e_4)/2,\ 
\overline\chi =\overline\rho/2,\ 
\overline\zeta =\overline e_1 + (\overline e_2+\overline e_4 + \overline a +\overline \omega_1)/2,\ 
\overline\eta =\overline e_1 + (\overline e_2+\overline e_4 + \overline a +\overline \omega_2)/2$;
to extend the action to $H^2(\tilde Y, \mathbb Z)$,  add also $y_1,\dots , y_4$ (see \eqref{y}).
\item The map $\pi_{2}^*$ annichilates $N$, and acts on  $\pi_{2*}W\subset \pi_{2*}\Lambda_{\mathrm{K3}}$ as\\
\begin{center}\xymatrix@R=0.01pc @C=0.05pc{
{\pi_{2}}^*: &D_4^{\oplus 2}  \ar@{}[r]|(.4){\oplus} &A_1(2)^{\oplus 2}  \ar@{}[r]|(.6){\oplus} &\langle -8\rangle \ar@{}[r]|(.35){\oplus} &{\begin{bmatrix}
8 & 0 &  0 \\
0 & 4 &  0 \\
0 & 0 &  4\end{bmatrix}} &\ \ar[rr] &\ &{\quad D_4^{\oplus{4}}} \ar@{}[r]|{\oplus} &A_1^{\oplus 2}  \ar@{}[r]|{\oplus} &\langle -4\rangle \ar@{}[r]|(.35){\oplus} &{\begin{bmatrix}
4 & 0 &  0 \\
0 & 2 &  0 \\
0 & 0 &  2\end{bmatrix}}\\
\ &{\begin{matrix}\hat e_1, \hat f_1\\ \hat e_2, \hat f_2\\ \hat e_3, \hat f_3\\ \hat e_4, \hat f_4\end{matrix}} &{\hat a_1, \hat a_2} &\hat\sigma &{\hat\rho,\ \hat\omega_1, \ \hat\omega_2} &\ \ar@{|->}[rr] &\ &{\begin{matrix}e_1+g_1, f_1+h_1\\ e_2+g_2, f_2+h_2\\ e_3+g_3, f_3+h_3\\ e_4+g_4, f_4+h_4\end{matrix}}  &{2a_1, 2a_2}  &2\sigma &{2\rho,2\omega_1,2\omega_2}
}\end{center}
Its action can be extended to $\pi_{2*}\Lambda_{\mathrm{K3}}$ adding the following elements (and their respective image to the image lattice):
$\hat\alpha=(\hat\rho+\hat\sigma)/4+(\hat e_1+\hat e_2+\hat f_1+\hat f_4)/2,\ 
\hat\beta=(\hat\rho-\hat\sigma)/4+(\hat e_1+\hat e_4+\hat f_1+\hat f_2)/2,\ 
\hat\varepsilon=\hat\rho/2+\hat f_2+\hat f_4+(\hat a_1+\hat a_2)/2,\ 
\hat\chi=(\hat\rho+\hat\sigma)/2,\ 
\hat\zeta=(\hat e_1+\hat f_1+\hat e_4+\hat f_2+\hat a_1+\hat\omega_1)/2,\ 
\hat\eta=(\hat e_1+\hat f_1+\hat e_4+\hat f_2+\hat a_2+\hat\omega_2)/2$;
to extend the action to $H^2(\tilde Z, \mathbb Z)$, add also $z_1,\dots , z_6$ (see \eqref{z}).
\item Recall that $H^2(\tilde Y, \mathbb Z)$ as an overlattice of finite index of $\widehat{\pi_{2}}_*H^2(\tilde Z, \mathbb Z)\oplus N_Y$. The lattice $N_Y\subseteq H^2(\tilde Y,\mathbb Z)$ is annichilated by $\widehat{\pi_{2}}^*$, and for the generators of $\widehat{\pi_{2}}_*H^2(\tilde Z, \mathbb Z)$ the map $\widehat{\pi_{2}}^*$ is defined as follows:
\begin{align*}
&\widehat{\pi_{2}}^*\overline e_i=\hat e_i+\hat f_i\ \mathrm{(for}\ i=1,\dots , 4,\ \mathrm{but\ }\overline e_2\ \mathrm{is\ not\ needed\ as\ generator)},\\
&\widehat{\pi_{2}}^*\overline\zeta=\hat e_1+\hat f_1+(\hat e_2+\hat f_2+\hat e_4+\hat f_4+\hat a_1+\hat a_2+2\hat\omega_1)/2,\\
&\widehat{\pi_{2}}^*\overline\eta=\hat e_1+\hat f_1+(\hat e_2+\hat f_2+\hat e_4+\hat f_4+\hat a_1+\hat a_2+2\hat\omega_2)/2,\\
&\widehat{\pi_{2}}^*\overline\alpha=\hat\alpha+\hat\beta,\quad \widehat{\pi_{2}}^*\overline a=\hat a_1+\hat a_2,\quad \widehat{\pi_{2}}^*\overline\chi=\widehat{\pi_{2}}^*\overline\rho/2=\hat\rho,\\
&\widehat{\pi_{2}}^*\widehat{\pi_{2}}_*n_6=2n_6,\quad \widehat{\pi_{2}}^*\widehat{\pi_{2}}_*\nu=2\nu,\\
&\widehat{\pi_{2}}^*\widehat{\pi_{2}}_*z_1=(\hat\alpha+\hat\beta+\hat e_1+\hat f_1+\hat e_2+\hat f_2+\hat a_1+\hat a_2+\widehat{\pi_{2}}^*\overline\eta+2n_2+n_5+n_8)/2,\\
&\widehat{\pi_{2}}^*\widehat{\pi_{2}}_*z_2=(\hat\rho+\hat e_2+\hat f_2+\hat e_4+\hat f_4+\hat a_1+\hat a_2+n_3+n_4+n_5+n_8)/2,\\
&\widehat{\pi_{2}}^*\widehat{\pi_{2}}_*z_5=(\hat\alpha+\hat\beta+\hat e_1+\hat f_1+\hat e_2+\hat f_2+\hat\rho+\widehat{\pi_{2}}^*\overline\eta+2n_6+n_5+n_8)/2+\hat\rho,\\
&\widehat{\pi_{2}}^*\widehat{\pi_{2}}_*z_6=(\hat\alpha+\hat\beta+\hat e_1+\hat f_1+\hat e_4+\hat f_4+\hat\rho+\widehat{\pi_{2}}^*\overline\zeta+2n_7+n_5+n_8)/2+\\
&\qquad \qquad \ +\hat a_1+\hat a_2+\hat e_2+\hat f_2+\hat\rho.
\end{align*}
Notice that $\widehat{\pi_{2}}^*\widehat{\pi_{2}}_*z_i=z_i+\hat\tau^*z_i$, and that to obtain the whole image of $\widehat{\pi_{2}}^*$ the images of the elements $y'_i$ of \eqref{y'} are also to be considered.
\end{enumerate}
\end{proposition}

\begin{proof}
We are going to prove only that the map $\pi_4^*$ acts on $\pi_{4*}W$ as stated above; the other cases are similar.\\
Since $\pi_4^*$ and $\pi_{4*}$ are dual maps, $\pi_4^*a=b$ iff $(b\cdot c)_X= (a\cdot\pi_{4*}c)_Y$ for every ${a\in\pi_{4*}W,}\ {c\in W}$: hence $\pi_4^*a\cdot c=a\cdot\pi_{4*}c$.
Take $e_j=(e_j,0,0,0)\in D_4^{\oplus 4}$: then $\pi_4^*\overline e_i\cdot e_j=\overline e_i\cdot \pi_{4*}e_j=\overline e_i\cdot \overline e_j$, but it holds also $\overline e_i\cdot\overline e_j=\overline e_i\cdot \pi_{4*}f_j=\overline e_i\cdot \pi_{4*}g_j=\overline e_i\cdot \pi_{4*}h_j$;
therefore $\pi_4^{*}\overline e_i=e_i+f_i+g_i+h_i$.
Take $a_1=(a_1,0)\in A_1^{\oplus 2}$: then $\pi_4^*\overline a\cdot a_1=\overline a\cdot \pi_{4*}a_1=2(\overline a\cdot \overline a)$ because $\pi_{4*}$ doubles the intersection form on $A_1\oplus A_1$; moreover, $\pi_{4*}a_1=\overline a=\pi_{4*}a_2$, therefore we get $\pi_4^{*}\overline a=2a_1+2a_2$.
Similarly, since $\pi_{4*}$ multiplies by 4 the intersection form of the sublattice of $W$ invariant for the action of $\tau$, we can conclude that $\pi_4^{*}\overline\rho=4\rho_1, \, \pi_4^{*}\overline\omega_1=4\omega_1,\pi_4^{*}\overline\omega_2=4\omega_2$.
\end{proof}

\begin{corollary}
The image of $H^2(\tilde Y, \mathbb Z)$ via the map $\pi_4^{*}$ coincides with the invariant lattice $\Lambda_{\mathrm{K3}}^{\langle\tau\rangle}$ described in Section \ref{Omega4}. In other words, it holds
\[\pi_4^{*}H^2(\tilde Y, \mathbb Z)=\Omega_4^{\perp_{\Lambda_{\mathrm{K3}}}}.\]
Similarly, we obtain:
\begin{align*}
\pi_2^{*}H^2(\tilde Z, \mathbb Z)=\Omega_2^{\perp_{H^2(X, \mathbb Z)}}=\Omega_2^{\perp_{\Lambda_{\mathrm{K3}}}},\\
\widehat{\pi_2}^{*}H^2(\tilde Y, \mathbb Z)=\Omega_2^{\perp_{H^2(\tilde Z, \mathbb Z)}}=\Omega_2^{\perp_{\Lambda_{\mathrm{K3}}}}.
\end{align*}
\end{corollary}
\begin{proof}
It holds $\pi_4^{*}\overline e_i=e_i+f_i+g_i+h_i$ for $i=1,\dots , 4$, $\pi_4^{*}\overline a=2(a_1+a_2)$, $\pi_4^{*}\overline\rho=4\rho$, $\pi_4^{*}\overline\omega_1=4\omega_1$, $\pi_4^{*}\overline\omega_2=4\omega_2$; however, these elements generate only a sublattice of finite index of $\pi_4^{*}H^2(\tilde Y, \mathbb Z)$: adding as generators the images via $\pi_4^{*}$ of the elements $\overline\alpha, \overline\chi, \overline\zeta, \overline\eta$ and $y_1,\dots , y_4$, we obtain the whole invariant lattice for the action of $\tau$ on $\Lambda_{\mathrm{K3}}$.
\end{proof}

\subsection{Characterization of the surface $\tilde Z$: the non-projective case} 

Nikulin's seminal work \cite{Nikulin2} provides a lattice theoretic characterization of K3 surfaces $X$ that admit a symplectic action of a cyclic group $G=\mathbb Z/n\mathbb Z$ (for $n=2,\dots , 8$), and of surfaces $\tilde Y$ that are the resolution of the quotient $X/G$, by providing a relation between some (even, negative definite) lattices $\Omega_G$ and $M_G$ that have to be primitively embedded in their respective Néron-Severi lattices; we want to show that similarly, in the case $G=\mathbb Z/4\mathbb Z$ (generated by an automorphism $\tau$ of $X$), the lattice $\Gamma$ characterizes the surface $\tilde Z$ that is the resolution of $Z\coloneqq X/\tau^2$. For simplicity, we are going to state our result for the most general K3 surface $X$: in this case, both $NS(X)=\Omega_4$ and $ NS(\tilde Y)=M$ have rank 14.

\begin{theorem}\label{Znonproj}
Let $\tilde Z$ be a K3 surface such that $rk(NS(\tilde Z))=14$. There exists a pair $(X,\tau)$ where $X$ is a K3 and $\tau$ is a symplectic automorphism of order 4 such that $\tilde Z$ is birationally equivalent to the resolution of the quotient $X/\tau^2$ if and only if $NS(\tilde Z)=\Gamma$ (see Def. \ref{def:Gamma}).
\end{theorem}
\begin{proof}
The ``only if'' is true by construction (see Section \ref{sec:resolutionZ}). Conversely,
suppose $NS(\tilde Z)=\Gamma$. The embedding $\Omega_2\subset\Gamma$ described in Remark \ref{Omega2suZ} defines a symplectic involution $\hat\tau$ on $\tilde Z$, and the Néron-Severi lattice of the resolution $\widetilde{\tilde Z/\hat\tau}$ is naturally a copy of $M$, as proved in Corollary \ref{MinquozienteZ}; therefore, by the results of Nikulin the surface $\widetilde{\tilde Z/\hat\tau}$ is the resolution of the quotient of a K3 surface $X$ for a symplectic automorphism $\tau$ of order 4, and it holds $NS(X)=\Omega_4$. The action of $\tau$ on $\Omega_4$ naturally defines a copy of $\Omega_2\subset\Omega_4$ by $\Omega_2=(\Omega_4^{\tau^2})^{\perp_{\Omega_4}}$, as described in Section \ref{Omega4}; taking the quotient map $\pi_{2}: X\rightarrow X/\tau^2$ and the resolution $\widetilde{X/\tau^2}$, it holds $NS(\widetilde{X/\tau^2})\simeq NS(\tilde Z)$.
\end{proof}

\section{Projective K3 surfaces with a symplectic automorphism of order 4 and their quotients}\label{sec:proj}

It was already known by Nikulin that the correspondence between surfaces $X$ that admit a symplectic action of an abelian group $G$, and surfaces $\tilde Y$ that are the resolution of $X/G$, is actually a moduli spaces correspondence (\cite{Nikulin2}, Prop. 2.9); the same idea was later generalized to the non-abelian case by Whitcher (\cite{Whitcher}, \S 3).\\
We can therefore refine the characterization of $X$, $\tilde Z$ and $\tilde Y$ by their Néron-Severi to the projective case.
The approach we follow mimics the one used in \cite{VGS},\cite{GS} for symplectic involutions, and in \cite{GP} for symplectic automorphisms of order 3. 

\begin{definition}[\cite{Dolgachev}.\S 1]\label{def:latticepolarized}
Let $\Lambda$ be an even lattice. A K3 surface $X$ is \emph{$\Lambda$-polarized} if there is a primitive embedding $\Lambda\hookrightarrow NS(X)$. The moduli space of $\Lambda$-polarized K3 surfaces has dimension $20-rk(\Lambda)$.
\end{definition}

\begin{proposition}[see \cite{Nikulin2}, Prop. 2.9; also \cite{VGS}, Prop. 2.2]\label{condizione_soprareticoli}
Projective K3 surfaces $X$ that admit a symplectic action of an abelian group $G$ are polarized with a lattice of rank $1+rk(\Omega_G)$ that contains primitively both the lattice $\Omega_G$ and an ample class $L$ of square $2d$, and projective K3 surfaces that are the resolution of $X/G$ are polarized with a lattice of rank $1+rk(M_G)$ that contains primitively both the lattice $M_G$ and an ample class $L$ of square $2d$. \\
Moreover, from Theorem \ref{Znonproj} we deduce that in the projective case $\tilde Z$ is polarized with a lattice of rank 15 that contains primitively both the lattice $\Gamma$ and an ample class $L$ of square $2d$. 
\end{proposition}

\begin{remark}\label{cor_condizione_soprareticoli}
Let $S$ be either $\Omega_G$ or $M_G$. The only lattices that satisfy the proposition above are $S\oplus\langle 2d\rangle$ and its cyclic overlattices of finite index (see \cite{GS1}, Prop. 6.1). What has to be determined is which of these actually exist, and wether the existing ones admit (one or more non-isomorphic) primitive embeddings in $\Lambda_{\mathrm{K3}}$: with each admissible lattice, and each non-isomorphic primitive embedding, we select a different irreducible component of the moduli space of projective K3 surfaces $X$ that admit a symplectic action of $G$ (if $S=\Omega_G$) or $\tilde Y$ that are the minimal resolution of $X/G$ (if $S=M_G$); the existence of a bijection between 
the irreducible components of the moduli spaces of $X$ and $\tilde Y$ is already known for $G=\mathbb Z/2\mathbb Z,\mathbb Z/3\mathbb Z$ (\cite{VGS}, \cite{GP}).
\end{remark}
\begin{remark}
If $G=\mathbb Z/4\mathbb Z$, we have to study $S=\Omega_G,\ M_G,\ \Gamma$. 
Recall also from Definition \ref{def:latticepolarized} that the moduli spaces of projective K3 surfaces $X$ that admit a  symplectic automorphism $\tau$ of order 4, and projective K3 surfaces are the resolution of $X/\tau^2$ or $X/\tau$ all have dimension 5.
\end{remark}

\begin{notation}
Consider the lattice $S\oplus \langle k \rangle$, where $S$ is one between $\Omega_4$, $\Gamma$ and $M$, and $\langle k \rangle$ is an even positive definite lattice of rank 1 and intersection matrix $[k]$.\\
Denote $(S\oplus \langle k \rangle)'$ and $(S\oplus \langle k \rangle)^{\star}$ any cyclic overlattices of $S\oplus \langle k \rangle$ such that it holds:
\[\frac{(S\oplus \langle k \rangle)'}{S\oplus \langle k \rangle}\simeq\frac{\mathbb Z}{2\mathbb Z}, \qquad \frac{(S\oplus \langle k \rangle)^{\star}}{S\oplus \langle k \rangle}\simeq\frac{\mathbb Z}{4\mathbb Z}.\]
\end{notation}

\subsection{Families of projective K3 surfaces with a symplectic automorphism $\tau$\\ of order 4}\label{sec:families4} 

\begin{definition}\label{reldeq}
Consider an even lattice $S$, its group of isometries $O(S)$ and its discriminant group $A_S$ with discriminant form $q_S$. We define two equivalence relations on $A_S$:
\begin{itemize}
\item \emph{by order and square}: two elements $r,s \in A_S$ are in relation ($r\sim_S s$) if they have the same order and square, i.e. $\langle r\rangle=\langle s\rangle=\mathbb Z/k \mathbb Z\subset A_S$ and $q_S(r)=q_S(s)=g\in\mathbb Q/2\mathbb Z$; we will denote the equivalence classes for this relation with the pair $(k,g)$;
\item \emph{by (induced) isometry}: two elements $r,s \in A_S$ are in relation ($r\approx_S s$) if there exists an isometry $\overline\varphi\in O(A_S)$ induced by an isometry $\varphi\in O(S)$ such that $\overline\varphi(r)=s$; we will denote the equivalence classes for this relation with the triple $(k,g,n)$, where $k,g$ are as above, and $n$ is the cardinality of the class (in our case, this is sufficient to uniquely identify each of them).
\end{itemize}
\end{definition}

\begin{proposition}\label{classi}
The relation $\sim_{\Omega_4}$ divides $A_{\Omega_4}$ in 7 non-trivial equivalence classes (plus the trivial one $\{0\}$), whose cardinality is displayed below. Each of them corresponds to an equivalence class for $\approx_{\Omega_4}$, except for $(2,1)$ which is the union of two classes: $(2,1,6)$ e $(2,1,10)$.
\emph{\begin{center}\begin{tabular}{c|c|c|c|c}
\diagbox{$k$}{$g$}&
 0 & 1/2 & 1 & 3/2 \\
\hline
2 & 15 & 32 & 16 & 0\\
4 & 240 & 240 & 240 & 240
\end{tabular}\end{center}}
\end{proposition}
\begin{proof}
The equivalence classes for $\sim_{\Omega_4}$ can be computed from a basis of $A_{\Omega_4}$ and its discriminant form.
The generators of $O(\Omega_4)$ can be computed using the Integral Lattices package in SAGE \cite{sage}: then, choosing for each of the classes $(k,g)$ of $\sim_{\Omega_4}$ a representative element $x_{(k,g)}$, their orbit for the induced action of $O(\Omega_4)$ on $A_{\Omega_4}$ is computed recursively (see Algorithm I.4 in \cite{Hulpke}). 
\end{proof}

\begin{corollary}\label{rappresentanti}
We give a representative element $x_{(k,g,n)}$ for each non-trivial equivalence class $(k,g,n)$ for the relation $\approx_{\Omega_4}$, in terms of the generators of $\Omega_4$ introduced in Section \ref{Omega4}.
\begin{longtable}{|c|c|}
\hline
class $(k,g,n)$ & representative $x_{(k,g,n)}$\TBstrut\\
\hline
 $(2,0,15)$ & $\frac{3e_3-f_3-g_3-h_3}{2}$\TBstrut\\
\hline
 $(2,1/2,32)$ & $\frac{2e_1-f_1-g_1+e_4-f_4+\alpha-\gamma+\sigma+(e_2-g_2+e_4-g_4+a_1-a_2+\sigma)/2}{2}$\TBstrut\\
\hline
 $(2,1,6)$ & $\frac{\sigma}{2}$\TBstrut\\
\hline
 $(2,1,10)$ &$\frac{\sigma+3e_3-f_3-g_3-h_3}{2}$\TBstrut\\
\hline
 $(4,0,240)$ &$\frac{3(3e_3-f_3-g_3-h_3)+2(2e_1-f_1-g_1+e_4-f_4+\alpha-\gamma) +e_2-g_2+e_4-g_4+a_1-a_2+\sigma}{4}$\TBstrut\\
\hline
 $(4,1/2,240)$ &$\frac{2(e_1-g_1+e_2-g_2)+a_1-a_2+\sigma}{4}$\TBstrut\\
\hline
 $(4,1,240)$ &$\frac{3(3e_1-f_1-g_1-h_1+3\alpha-\beta-\gamma-\delta)+2(e_2-g_2+e_4-g_4+\sigma)}{4}$\TBstrut\\
\hline
 $(4,3/2,240)$ &$\frac{3(3e_1-f_1-g_1-h_1)}{4}$\TBstrut\\
\hline
\end{longtable}
\end{corollary}

\begin{theorem}\label{NSpossibili}
Let $X$ be a projective K3 surface that admits a symplectic automorphism of order 4, such that $rk(NS(X))=15$. Then, $NS(X)$ is one of the following lattices:
\begin{enumerate}
\item for every $d\in \mathbb N$, $NS(X)=\Omega_4\oplus\langle 2d \rangle$;
\item for $d\neq_4 1$, $NS(X)=(\Omega_4\oplus\langle 2d \rangle)'$; for $d=_4 2$ there are two non isometric possibilities.
\item For $d=_4 0$, $NS(X)=(\Omega_4\oplus\langle 2d \rangle)^{\star}$.
\end{enumerate}
\end{theorem}
\begin{proof}
Overlattices of index 2 of $\Omega_4\oplus\langle 2d \rangle$ correspond to isotropic elements in $A_{\Omega_4\oplus\langle 2d \rangle}$ of the form $(L+v)/2$, where $L$ generates $\langle 2d \rangle$ and $v\in\Omega_4$. Requiring
\[\bigg(\frac{L+v}{2}\bigg)^2=\frac{d}{2}+\bigg(\frac{v}{2}\bigg)^2=0 \quad \mathrm{in} \quad \frac{\mathbb Q}{2\mathbb Z}\]
we see that to each value of $d$ modulo $4$ corresponds an isometry class of order 2 of the ones described in Proposition \ref{classi}, so for $d=4h+1$ no overlattice of index 2 of $\Omega_4\oplus\langle 2d \rangle$ exists; on the other hand, for $d=4h+2$ there are two equivalence classes for $\approx_{\Omega_4}$ that can be used: they indeed produce two overlattices of $\Omega_4\oplus\langle 2d \rangle$ that are not in the same genus, as can be seen by constructing them explicitely as in Example \ref{esempiL}. \\
We proceed similarly for the overlattices of type $(\Omega_4\oplus\langle 2d \rangle)^{\star}$.
\end{proof}

\begin{theorem}\label{unicita}
Each of the lattices presented in Theorem \ref{NSpossibili} admits a unique primitive embedding in $\Lambda_{\mathrm{K3}}$ up to isometries of $\Lambda_{\mathrm{K3}}$.
\end{theorem}
\begin{proof}
Let $NS(X)=\Omega_4\oplus\langle 2d\rangle$, let $T(X)=NS(X)^{\perp_{\Lambda_{\mathrm{K3}}}}$; it holds $\lambda(A_{T(X)})=\lambda(A_{NS(X)})=(2,2,4,4,4,4, 2d)$ ($(2,2,2,4,4,4,4)$ for $d=1$); if $d$ is odd, then $\mathbb Z/2d\mathbb Z=\mathbb Z/2\mathbb Z\times \mathbb Z/d\mathbb Z$, so we have $A_{NS(X)}=(\mathbb Z/2\mathbb Z)^3\times A'$ and $NS(X)$ satisfies Corollary \ref{condizione3}; if $d$ is even, then $\lambda(A_{T(X)})=(2,2,4,4,4,4, 4d')$, so $T(X)$ satisfies Proposition \ref{condizione1}, and therefore $NS(X)$ satisfies Theorem \ref{unicitaprimitivi}.\\
Let $NS(X)=(\Omega_4\oplus\langle 2d\rangle)'$, let $T(X)=NS(X)^{\perp_{\Lambda_{\mathrm{K3}}}}$: for $d=_4 2,3$ $A_{NS(X)}$ has length 5, so there exists a unique primitive embedding of $NS(X)$ in $\Lambda_{\mathrm{K3}}$ thanks to Theorem \ref{condizione2}; for $d=4d'$, $T(X)$ satisfies Proposition \ref{condizione1}, for $\lambda(A_{T(X)})=(2,2,2,2,4,4, 8d')$; therefore $NS(X)$ satisfies Theorem \ref{unicitaprimitivi}.\\
Let $NS(X)=(\Omega_4\oplus\langle 2d\rangle)^{\star}$: then $A_{NS(X)}$ has length 5, so it satisfies Theorem \ref{condizione2}.
\end{proof}

\begin{example}\label{esempiL} We are going to exhibit a primitive embedding in $\Lambda_{\mathrm{K3}}$ of each of the lattices presented in \ref{NSpossibili}, having fixed the primitive embedding of $\Omega_4$ in $H^2(X,\mathbb Z)$ described in Section \ref{Omega4}, by providing examples of primitive classes  $L\in \Omega_4^{\perp_{H^2(X,\mathbb Z)}}$ such that $L^2=2d$. We use the notation of Section \ref{marking}, and we construct the overlattices using the elements $x_{(k,g,n)}$ in Corollary \ref{rappresentanti}.
\begin{enumerate}
\item For every $d\in\mathbb N\setminus\{0\}$, the class 
\[L_0=L_0(d)=\frac{a_1+a_2+\omega_1+\omega_2}{2}+d\bigg(\frac{-a_1-a_2+\omega_1-\omega_2}{2}\bigg)\]
generates the lattice $\langle 2d\rangle$ such that $\Omega_4\oplus\langle 2d\rangle$ is primitively embedded in $H^2(X,\mathbb Z)$.
\item For $d=4(h-1),\ h\in\mathbb N\setminus\{0,1\}$ the class
\[L_{2,0}(h)=2L_0(h)+e_3+f_3+g_3+h_3\]
generates the lattice $\langle 2d\rangle$ such that $(\Omega_4\oplus\langle 2d \rangle)'$ is primitively embedded in $H^2(X,\mathbb Z)$; $L_{2,0}/2+x_{(2,0,15)}$ is in fact integer in $H^2(X,\mathbb Z)$.
\item For $d=4h+2,\ h\in\mathbb N$, the class
\[L_{2,2}^{(1)}(h)=2L_0(h)+\rho\]
generates the lattice $\langle 2d\rangle$ such that $(\Omega_4\oplus\langle 2d \rangle)'$ is primitively embedded in $H^2(X,\mathbb Z)$; $L_{2,2}^{(1)}/2+x_{(2,1,6)}$ is in fact integer in $H^2(X,\mathbb Z)$.
\item For $d=4(h-1)+2,\ h\in\mathbb N\setminus\{0\}$, the class
\[L_{2,2}^{(2)}(h)=2L_0(h)+\rho+e_3+f_3+g_3+h_3\]
generates the lattice $\langle 2d\rangle$ such that $(\Omega_4\oplus\langle 2d \rangle)'$ is primitively embedded in $H^2(X,\mathbb Z)$; $L_{2,2}^{(2)}/2+x_{(2,1,10)}$ is in fact integer in $H^2(X,\mathbb Z)$.
\item For $d=4h+3,\ h\in\mathbb N$, the class
\[L_{2,3}(h)=2L_0(h)+\omega_2+\frac{e_1+f_1+g_1+h_1+e_4+f_4+g_4+h_4+a_1+a_2+3\rho}{2}\]
generates the lattice $\langle 2d\rangle$ such that $(\Omega_4\oplus\langle 2d \rangle)'$ is primitively embedded in $H^2(X,\mathbb Z)$; $L_{2,3}/2+x_{(2,1/2,32)}$ is in fact integer in $H^2(X,\mathbb Z)$.
\item For $d=16(h-1),\ h\in\mathbb N\setminus\{0,1\}$, the class
\[L_{4,0}(h)=4L_0(h)+e_1+f_1+g_1+h_1+3(e_3+f_3+g_3+h_3)+e_4+f_4+g_4+h_4+a_1+a_2+\rho+2\omega_2\]
generates the lattice $\langle 2d\rangle$ such that $(\Omega_4\oplus\langle 2d \rangle)^{\star}$ is primitively embedded in $H^2(X,\mathbb Z)$; $L_{4,0}/4+x_{(4,0,240)}$ is in fact integer in $H^2(X,\mathbb Z)$.
\item For $d=16h+4,\ h\in\mathbb N$, the class
\[L_{4,4}(h)=4L_0(h)+a_1+a_2+\rho+2\omega_2\]
generates the lattice $\langle 2d\rangle$ such that $(\Omega_4\oplus\langle 2d \rangle)^{\star}$ is primitively embedded in $H^2(X,\mathbb Z)$; $L_{4,4}/4+x_{(4,3/2,240)}$ is in fact integer in $H^2(X,\mathbb Z)$.
\item For $d=16(h-4)+8,\ h\in\mathbb N\setminus\{0,1,2,3\}$, the class
\[L_{4,8}(h)=4L_0(h)+\rho+\frac{3(e_2+f_2+g_2+h_2)+7(e_4+f_4+g_4+h_4)}{2}\]
generates the lattice $\langle 2d\rangle$ such that $(\Omega_4\oplus\langle 2d \rangle)^{\star}$ is primitively embedded in $H^2(X,\mathbb Z)$; $L_{4,8}/4+x_{(4,1,240)}$ is in fact integer in $H^2(X,\mathbb Z)$.
\item For $d=16h+12,\ h\in\mathbb N$, the class
\[L_{4,12}(h)=4L_0(h)+e_1+f_1+g_1+h_1+2(e_4+f_4+g_4+h_4+\rho+\omega_1+\omega_2)\]
generates the lattice $\langle 2d\rangle$ such that $(\Omega_4\oplus\langle 2d \rangle)^{\star}$ is primitively embedded in $H^2(X,\mathbb Z)$; $L_{4,12}/4+x_{(4,1/2,240)}$ is in fact integer in $H^2(X,\mathbb Z)$.
\end{enumerate}
\end{example}

\subsection{Relations between the families associated to the action of $\tau$, and those associated to the action of $\tau^2$}

Let $X$ be a projective K3 surface which admits a symplectic automorphism of order 4 $\tau$, and suppose $NS(X)$ be as in Theorem \ref{NSpossibili}: then $X$ is a special member of one of the families of projective surfaces that admit a symplectic involution (as $\tau^2$ is such), whose general element has Picard number $9$. 
The lattice $NS(X)$ is an overlattice of finite index of ${R\oplus\Omega_2\oplus\langle 2d\rangle}$.

\begin{theorem}[see \cite{VGS}, Prop. 2.2]\label{thm:famiglie2}
Let $X$ be a projective K3 surface with a symplectic involution $\iota$, such that $rk(NS(X))=9$.
Then we have the following possible cases for $NS(X)$:
\begin{enumerate}[(a)]
\item for every $d$, $NS(X)=\Omega_2\oplus\langle 2d\rangle$;
\item for $d$ even, $NS(X)=(\Omega_2\oplus\langle 2d\rangle)'$.
\end{enumerate}
\end{theorem}

To determine which of these families our Néron-Severi groups belong to, we return to the setting of Example \ref{esempiL}, and we consider the embedding of $\Omega_2$ in $\Omega_4$ as in Section \ref{Omega4}.

\begin{theorem}\label{NSXordine2}
\begin{enumerate}
\item For every $d\in \mathbb N$, $NS(X)=\Omega_4\oplus\langle 2d \rangle$ corresponds to case (a) of Thm \ref{thm:famiglie2}.
\item For $d=_4 2,3$, $NS(X)=(\Omega_4\oplus\langle 2d \rangle)'$ corresponds to case (a) of Thm \ref{thm:famiglie2}.
\item For $d=_4 0$, $NS(X)=(\Omega_4\oplus\langle 2d \rangle)'$ corresponds to case (b) of Thm \ref{thm:famiglie2}.
\item For $d=_4 0$, $NS(X)=(\Omega_4\oplus\langle 2d \rangle)^{\star}$ corresponds to case (b) of Thm \ref{thm:famiglie2}
\end{enumerate}
The following table describes the situation:
\begin{center}
\begin{tabular}{|c|c|c|c|c|c|}
\hline
\multicolumn{2}{|c|}{$NS(X)$\Tstrut} 	& $L$ glues to $\Omega_2$ \Tstrut  & $L$ glues to $R$ \Tstrut &  $NS(X)/(\Omega_2\oplus R\oplus\langle 2d \rangle)$ \\ [6pt]
\hline
$\forall d$\Tstrut	& $\Omega_4\oplus\langle 2d\rangle$\Tstrut & No & No & $(\mathbb Z/2\mathbb Z)^4$\\ [6pt]
\hline
$d=_4 2, 3$\Tstrut & $(\Omega_4\oplus\langle 2d\rangle)'$\Tstrut  & No & Yes & $(\mathbb Z/2\mathbb Z)^5$\\ [6pt]
\hline
\multirow{2}{*}{$d=_4 0$\Tstrut} & $(\Omega_4\oplus\langle 2d\rangle)'$\Tstrut & Yes & Yes & $(\mathbb Z/2\mathbb Z)^5$\\[6pt]
\cline{2-5} 
\ & $(\Omega_4\oplus\langle 2d\rangle)^{\star}$\Tstrut  & Yes & Yes & $(\mathbb Z/2\mathbb Z)^4\times \mathbb Z/4\mathbb Z$   \\ [6pt]
\hline
\end{tabular}
\end{center}
\end{theorem}
\begin{proof}
The class $L_0$ doesn't glue to $\Omega_4$, so it cannot glue neither to $\Omega_2$ nor to $R$.
Since $L_{2,3}^2=2d$ with $d=_4 3$, this case corresponds necessarily to case (a) of Thm \ref{thm:famiglie2}; $\Omega_2\oplus R\oplus\langle 2d \rangle$ has index  $2^5$ in $NS(X)$, because there exists $r\in R$ such that $(L_{2,3}+r))/2\in NS(X)$: $r=\sigma+ g_1-f_1+e_4-f_4+\alpha-\gamma-(e_2-g_2+e_4-g_4+a_1-a_2+\sigma)/2$.
Despite $d$ being even, there are no elements in $\Omega_2$ that glue to $L_{2,2}^{(1)}$ or $L_{2,2}^{(2)}$, so we are again in case (a): this means that a symplectic involution that satisfies case (b) of Theorem \ref{thm:famiglie2}, with $L^2=_8 4$, cannot be the square of a symplectic automorphism of order 4. \\
The gluings for the cases in which $d=4k$ are described as follows:
$(L_{2,0}+\tilde v)/2,\ (L_{2,0}+r)/2\in NS(X)$ for $\tilde v=e_3-g_3+f_3-h_3 \in \Omega_2$, $r=e_3-f_3+g_3-h_3\in R$; since $(\tilde v+r)/2$ is one of the elements that glue $\Omega_2$ to $R$, the index of $\Omega_2\oplus R\oplus\langle 2d \rangle$ in $NS(X)$ is still $2^5$;
$(L_{4,0}+\tilde v)/2,\ (L_{4,0}+r)/2\in NS(X)$ for the same $\tilde v\in \Omega_2,\ r\in R$ as $L_{2,0}$;
$(L_{4,4}+\tilde v)/2,\ (L_{4,4}+r)/2\in NS(X)$ for $\tilde v=f_2-h_2+f_4-h_4\in \Omega_2$, $r=a_1-a_2+\sigma\in R$;
$(L_{4,8}+\tilde v)/2,\ (L_{4,8}+r)/2\in NS(X)$ for $\tilde v=\alpha-\gamma+\beta-\delta+e_1-g_1+f_1-h_1\in \Omega_2$, $r=e_1-f_1+g_1-h_1+\alpha-\beta+\gamma-\delta\in R$;
$(L_{4,12}+\tilde v)/2, (L_{4,12}+r)/2\in NS(X)$ for $\tilde v=e_1-g_1+f_1-h_1\in \Omega_2$, $r=e_1-f_1+g_1-h_1\in R$. 
Thus, these cases correspond to case (b) of Theorem \ref{thm:famiglie2}; $NS(X)$ is obtained as overlattice of $\Omega_2\oplus R\oplus\langle 2d \rangle$ by gluing firstly $\Omega_2$ with $R$ to get $\Omega_4$, and then $L_{4,k}$ with $\Omega_4$ as in Example \ref{esempiL}.
\end{proof}

\subsection{Families of projective K3 surfaces which arise as desingularization of $X/\tau$}

Projective surfaces $\tilde Y$ that are the resolution of $X/\tau$ have to primitively contain in their Néron-Severi both the exceptional lattice $M$ described in Section \ref{Y} (see \cite{Nikulin2}, \S 5-7), and a positive class $L$ of square $2e$ that generates $M^{\perp_{NS(\tilde Y)}}$: therefore, $\tilde Y$ is polarized with the lattice $M\oplus\langle 2e \rangle$ or one of its cyclic overlattices. 

\begin{theorem}\label{classiM}
The relation $\sim_{M}$ (see Def. \ref{reldeq}) divides $A_{M}$ in 7 non-trivial equivalence classes (plus the trivial one $\{0\}$):
\emph{\begin{center}\begin{tabular}{c|c|c|c|c}
\diagbox{$k$}{$g$}&
 0 & 1/2 & 1 & 3/2 \\
\hline
2 & 3 & 8 & 4 & 0\\
4 & 12 & 12 & 12 & 12
\end{tabular}\end{center}}
Each of them corresponds to an equivalence class for $\approx_{M}$, except for $(2,1)$ which is the union of two classes: $(2,1,1)$ and $(2,1,3)$.
We give a representative element $x_{(k,g,n)}$ for each non-trivial equivalence class $(k,g,n)$
 in terms of the generators of $M$ introduced in Section \ref{Y}.
\end{theorem}

\begin{longtable}{|c|c|}
\hline
class $(k,g,n)$ & representative $x_{(k,g,n)}$\TBstrut\\
\hline
 $(2,0,3)$ & $\frac{m^2_1+2m^2_2+3m^2_3+m^3_1+2m^3_2+3m^3_3}{2}$\TBstrut\\
\hline
 $(2,1/2,8)$ & $\frac{m^4_1+m^4_3+\tilde m^2}{2}$\TBstrut\\
\hline
 $(2,1,1)$ & $\frac{\tilde m^1+\tilde m^2}{2}$\TBstrut\\
\hline
 $(2,1,3)$ &$\frac{m^2_1+2m^2_2+3m^2_3+m^3_1+2m^3_2+3m^3_3+\tilde m^1+\tilde m^2}{2}$\TBstrut\\ 
\hline
 $(4,0,12)$ &$\frac{m^2_1+2m^2_2+3m^2_3+m^3_1+2m^3_2+3m^3_3}{4}+\frac{\tilde m^2}{2}$\TBstrut\\
\hline
 $(4,1/2,12)$ &$\frac{m^3_1+2m^3_2+3m^3_3+m^4_1+2m^4_2+3m^4_3}{4}+\frac{m^4_1+m^4_3}{2}$\TBstrut\\
\hline
 $(4,1,12)$ &$\frac{m^2_1+2m^2_2+3m^2_3+m^4_1+2m^4_2+3m^4_3}{4}+\frac{m^3_1+m^3_3+m^4_1+m^4_3+\tilde m^1}{2}$\TBstrut\\
\hline
 $(4,3/2,12)$ &$\frac{m^2_1+2m^2_2+3m^2_3+m^4_1+2m^4_2+3m^4_3}{4}+\frac{m^3_1+m^3_3}{2}$\TBstrut\\
\hline
\end{longtable}

\begin{theorem}\label{NSpossibiliY}
Let $\tilde Y$ be a projective K3 surface such that $rk(NS(\tilde Y))=15$ and $NS(\tilde Y)$ contains primitively $M$ and $\langle 2e \rangle,\ e\in\mathbb N\setminus\{0\}$. Then, $NS(\tilde Y)$ is one of the following:
\begin{enumerate}
\item for every $e$, $NS(\tilde Y)=M\oplus\langle 2e \rangle$;
\item for $e\neq_4 1$, $NS(\tilde Y)=(M\oplus\langle 2e \rangle)'$; for $e=_4 2$, there are two (non isometric) possibilities. 
\item For $e=_4 0$, $NS(\tilde Y)=(M\oplus\langle 2e \rangle)^{\star}$.
\end{enumerate}
Each of these lattices admits a unique primitive embedding in $\Lambda_{\mathrm{K3}}$.
\end{theorem}
\begin{proof}
The overlattices of $M\oplus\langle 2e \rangle$ are in bijection with the equivalence classes for $\approx_M$.
Fix the primitive embedding $M\hookrightarrow\Lambda_{\mathrm{K3}}$ as in Section \ref{Y}: since the orthogonal complement of $M$ is the lattice $\pi_{4*}H^2(X, \mathbb Z)$, we can use as generators of the lattice $\langle 2e\rangle$ the primitive classes $\overline L$ in $H^2(\tilde Y, \mathbb Z)$ obtained from $\pi_{4*}L$ (with $L$ as in  Example \ref{esempiL}) as follows.
Refer to Section  \ref{sec:maps} for the computation of $\pi_{4*}L$, and see also the corresponding classes $D_1$ in Section \ref{sec:eigenspaces} for the explicit gluing element $(\overline L+m)/k,\ m\in A_M,\ k=2,4$:
\begin{longtable}{|c|c|c|}
\hline
{$NS(\tilde Y)$\Tstrut}    		    & $e$\Tstrut &$\overline L$\\ [6pt]
\hline
\multirow{4}{*}{$M\oplus\langle 2e\rangle$\Tstrut} 	&$4(h-1)$ & $\pi_{4*}L_{4,0}(h)/4$\Tstrut\\[6pt] 
												&$4h+1$ & $\pi_{4*}L_{4,4}(h)/4$\\[6pt]
												&$4(h-4)+2$ & $\pi_{4*}L_{4,8}(h)/4$\\[6pt]
												&$4h+3$ & $\pi_{4*}L_{4,0}(h)/4$\Bstrut\\
\hline
\multirow{4}{*}{$(M\oplus\langle 2e\rangle)'$\TTstrut} 	&$4(h-1)$ & $\pi_{4*}L_{2,0}(h)/2$\Tstrut\\[6pt] 
													&$4h+2$ & $\pi_{4*}L_{2,2}^{(1)}(h)/2$\\[6pt]
													&$4(h-1)+2$ & $\pi_{4*}L_{2,2}^{(2)}(h)/2$\\[6pt]
													&$4h+3$ & $\pi_{4*}L_{2,3}(h)/2$\Bstrut\\
\hline
$(M\oplus\langle 2e\rangle)^{\star}$ &$4h$ &$\pi_{4*}L_0(h)$\TBstrut\\
\hline
\end{longtable}
Notice that for every choice of $e=_4 2$ there are two non isomorphic realizations of $(M\oplus\langle 2e\rangle)'$, using alternatively $\pi_{4*}L_{2,2}^{(i)}(h)/2,\ i=1,2$: indeed, for $h$ odd both of them glue to the class $(2,1,1)$, while for $h$ even they both glue to $(2,1,3)$. The resulting lattices belong to different genera.\\
Each of the possible lattices $NS(\tilde Y)$ admits a unique primitive embedding in $H^2(X,\mathbb Z)$, because $\ell(A_{NS(\tilde Y)})\leq 5$, so Theorem \ref{condizione2} holds.
\end{proof}

\begin{theorem}\label{relations}
There is a 1-1 correspondence between families of K3 surfaces $X$ with $NS(X)$ as in \ref{NSpossibili}, and families of K3 surfaces $\tilde Y$ with $NS(\tilde Y)$ as in Theorem \ref{NSpossibiliY}. The primitive classes $\overline L\in NS(\tilde Y)$ that generate the sublattices $\langle nd\rangle$ as stated are indicated in curly brackets.
\begin{center}
\begin{tabular}{|c|c|c c|}
\hline
\multicolumn{2}{|c|}{$NS(X)$\Tstrut}    		    &\multicolumn{2}{c|} {$NS(\tilde Y)$\Tstrut} \\ [6pt]
\hline
$\forall d$\Tstrut	& $\Omega_4\oplus\langle 2d\rangle$\Tstrut & $(M\oplus\langle 8d\rangle)^{\star}$ &$\{\overline L=\pi_{4*}L\}$\Tstrut\\ [6pt]
\hline
\multirow{2}{*}{$d=_4 2$\Tstrut} & $(\Omega_4\oplus\langle 2d\rangle)'^{(1)}$\Tstrut 
&$(M\oplus\langle 2d\rangle)'^{(1)}$ & $ \big\{\overline L=\frac{\pi_{4*}L}{2}\big\} $\Tstrut \\[6pt]
\ & $(\Omega_4\oplus\langle 2d\rangle)'^{(2)}$\Tstrut &$(M\oplus\langle 2d\rangle)'^{(2)} $ & $ \big\{\overline L=\frac{\pi_{4*}L}{2}\big\} $\Tstrut \\[6pt]
\hline
$d=_4 3$\Tstrut & $(\Omega_4\oplus\langle 2d\rangle)'$\Tstrut & $(M\oplus\langle 2d\rangle)'$ & $ \big\{\overline L=\frac{\pi_{4*}L}{2}\big\}$\Tstrut \\ [6pt]
\hline
\multirow{2}{*}{$d=_4 0$\Tstrut} & $(\Omega_4\oplus\langle 2d\rangle)'$\Tstrut 
&$(M\oplus\langle 2d\rangle)'$ & $\big\{\overline L=\frac{\pi_{4*}L}{2}\big\} $\Tstrut \\[6pt]
\ & $(\Omega_4\oplus\langle 2d\rangle)^\star$\Tstrut &$M\oplus\langle d/2\rangle $ & $ \big\{\overline L=\frac{\pi_{4*}L}{4}\big\} $\Tstrut \\[6pt]
\hline
\end{tabular}
\end{center}
\end{theorem}
\begin{proof}
The map $\pi_{4*}$ kills $\Omega_4$, so the possible Néron-Severi groups for the general smooth quotient surface  $\tilde Y=\widetilde{X/\tau}$ are determined by how the image of the ample class $L$ glues to the exceptional lattice $M$. For each of the $L$ in Example \ref{esempiL} we compute $\pi_{4*}L$ using the description of the image lattice $\pi_{4*}H^2(X,\mathbb Z)$ given in Section \ref{sec:maps}; we find the unique integral and primitive $\overline L=\pi_{4*}L/k$, where $k$ can be $1, 2$ or $4$ depending on the case, and we then compare $\overline L^2$ to $L^2$.
\end{proof}
 
\subsection{Families of projective K3 surfaces which arise as desingularization of $X/\tau^2$}

The process used in the previous section can be also applied to describe the K3 surfaces $\tilde Z$ that are resolution of $X/\tau^2$, and the relations between $NS(X)$ and $NS(\tilde Z)$; for the general symplectic involution, this was already done by Garbagnati and Sarti: 
\begin{theorem}[see \cite{GS}, Cor. 2.2]\label{NSZ}
Let $X$ be an algebraic K3 surface with $rk(NS(X))=9$ admitting a Nikulin involution $\iota$, and let $\tilde Z$ be the desingularization of the quotient $X/\iota$. Then:
\begin{enumerate}[(a)]
\item $NS(X) = \Omega_2\oplus\langle 2d\rangle$ if and only if $NS(\tilde Z) = (N\oplus\langle 4d\rangle)'$;
\item $NS(X) = (\Omega_2\oplus\langle 2d\rangle)'$ if and only if $NS(\tilde Z) =N\oplus\langle d\rangle$.
\end{enumerate}
\end{theorem}

\begin{theorem}
Let $\tilde Z$ be a K3 surface such that $rk(NS(\tilde Z))=15$; suppose $NS(\tilde Z)$ admits a primitive embedding of both $\Gamma$ (see Def. \ref{def:Gamma}) and of a class of positive square $2d$ that generates $\Gamma^{\perp_{NS(\tilde Z)}}$. Then $d=2x$, and $NS(\tilde Z)$ is one of the following:
\begin{enumerate}
\item for every $x$, $NS(\tilde Z)=(\Gamma\oplus\langle 4x\rangle)'$, uniquely determined.
\item for $x=_4 2, 3$, $NS(\tilde Z)=(\Gamma\oplus\langle 4x\rangle)^\star$, uniquely determined.
\end{enumerate}
Moreover, there exists a unique primitive embedding of these lattices in $H^2(\tilde Z, \mathbb Z)$ up to isometries of the latter.
\end{theorem}
\begin{proof}
The lattice $\Gamma\oplus\langle 2d\rangle$ cannot be the Néron-Severi of a K3 surface, since $\lambda(A_{\Gamma\oplus\langle 2d\rangle})=(2,2,2,2,2,2,4,4,2d)$), so its length is $9>22-rk(\Gamma\oplus\langle 2d\rangle)$ (see Rem. \ref{nonesiste}).\\
Consider the table of non-trivial equivalence classes for $\sim_{\Gamma}$:
\begin{center}\begin{tabular}{c|c|c|c|c|c|c|c|c}
\diagbox{$k$}{$g$}&
 0 & 1/4 & 1/2 & 3/4 & 1 & 7/4 & 3/2 & 9/4\\
\hline
2 &  127 & 0 & 0 & 0 & 128 & 0 & 0 & 0\\
4 &  0 & 256 & 144 & 0 & 0 & 256 & 112 & 0
\end{tabular}\end{center}
an element of the form $(E+\gamma)/2$, with $E^2=2d$ and $\gamma\in\Gamma$, has integer, even self-intersection only if $d$ is even, and an element of the form $(E+\gamma)/4$ only if $d=2x$ with $x=_4 2,3$. The non-trivial equivalence classes for $\approx_{\Gamma}$ are presented in the following table; the corresponding overlattice of $\Gamma\oplus\langle 4x\rangle$ can be realized having fixed the embedding $\Gamma\in H^2(\tilde Z,\mathbb Z)$ as in Def. \ref{def:Gamma}, using as positive class $\hat L=\pi_{2*}L$ for $L\in\{L_0, L_{2,2}^{(1)}, L_{2,2}^{(2)}, L_{2,3}\}$, and  $\hat L=\pi_{2*}L/2$ for $L\in\{L_{2,0}, L_{4,0}, L_{4,4}, L_{4,8},  L_{4,12}\}$.
\begin{longtable}{|c|c|c|}
\hline
class $(k,g,n)$ & representative $x_{(k,g,n)}$ & glues to: \TBstrut\\
\hline
$(2,0,1)$ & $(n_3+n_4+n_5+n_8)/2$ & $\hat L_{2,2}^{(i)}/2$ for $i=1,2$  \TBstrut\\
\hline
$(2,0,6)$ & $(n_5+n_6+n_7+n_8)/2$ & $\hat L_0(h)/2$, $h=_2 0$ \TBstrut\\ 
\hline
$(2,0,30)$  &$\begin{matrix}[1.4] \addlinespace[2pt]\frac{(\hat e_3-\hat f_3)+n_3+n_4}{2}\\
\frac{(\hat\alpha-\hat\beta)+(\hat e_1-\hat f_1)+n_3+n_4+n_5+n_8}{2} \end{matrix}$ &{$\begin{matrix}[1.4]\hat L_{4,0}/2\\ \hat L_{4,8}/2\end{matrix}$}\TBstrut\\ 
\hline
$(2,0,90)$ & $\frac{(\hat e_3-\hat f_3)+n_2+n_3+n_4+n_5+n_6+n_8}{2}$ & $\hat L_{2,0}(h)/2$, $h=_2 1$ \TBstrut\\ 
\hline
$(2,1,2)$ & $(n_5+n_8)/2$ & $\hat L_{2,3}/2$ \TBstrut\\
\hline
$(2,1,6)$ & $(n_2+n_3+n_4+n_5+n_6+n_8)/2$ & $\hat L_0(h)/2$, $h=_2 1$ \TBstrut\\ 
\hline
$(2,1,30)$ &$\begin{matrix}[1.4]\addlinespace[2pt]\frac{(\hat e_1-\hat f_1)+(\hat e_4-\hat f_4)+n_3+n_4}{2}\\
(\hat e_1-\hat f_1)/2\end{matrix}$&$\begin{matrix}[1.4]\hat L_{4,4}\\ \hat L_{4,12}\end{matrix}$ \TBstrut\\
\hline
$(2,1,90)$ & $\frac{(\hat e_3-\hat f_3)+n_5+n_6+n_7+n_8}{2}$ & $\hat L_{2,0}(h)/2$, $h=_2 0$\TBstrut\\ 
\hline
$(4,1/4,256)$ & $\frac{(\hat e_1-\hat f_1)+(\hat e_4-\hat f_4)+x_1+ x_2+n_2}{2}+\frac{3n_5+n_8}{4}$ &$\hat L_{2,3}(h)/4$, $h=_2 1$\TBstrut\\ 
\hline
$(4,1/2,24)$ & $\frac{x_1}{2}+\frac{3n_3+n_4+3n_5+n_8}{4}$ &$\hat L_{2,2}^{(1)}(h)/4$, $h=_2 1$\TBstrut\\ 
\hline
$(4,1/2,120)$ &$\frac{(\hat e_3-\hat f_3)+ x_1+n_2+n_7}{2}+\frac{n_3+3n_4+3n_5+n_8}{4}$ &$\hat L_{2,2}^{(2)}(h)/4$, $h=_2 0$\TBstrut\\ 
\hline
$(4,5/4,256)$ &$\frac{x_1+x_2+n_3+n_4+n_7+(\hat e_1-\hat f_1)+(\hat e_4-\hat f_4)}{2}+\frac{3n_5+n_8}{4}$ & $\hat L_{2,3}(h)/4$, $h=_2 0$\TBstrut\\ 
\hline
$(4,3/2,40)$ &$\frac{(\hat e_3-\hat f_3)+ x_1}{2}+\frac{3n_3+n_4+3n_5+n_8}{4}$  & $\hat L_{2,2}^{(2)}(h)/4$, $h=_2 1$\TBstrut\\
\hline
$(4,3/2,72)$ &$\frac{x_1+n_2+n_7}{2}+\frac{n_3+3n_4+3n_5+n_8}{4}$ & $\hat L_{2,2}^{(1)}(h)/4$, $h=_2 0$\TBstrut\\ 
\hline
\end{longtable}
Now, the classes $(2,0,1)$ and $(2,1,2)$ produce overlattices of $\Gamma\oplus\langle 4x\rangle$ that are not admissible as Néron-Severi of a K3 surfaces, because they have $\ell=9$ (see Rem. \ref{nonesiste}).
For the remaining classes $(k,g,n)$, those contained in the same equivalence class $(k,g)$ for the relation $\sim_{\Gamma}$ give rise to isomorphic lattices: indeed, having fixed $x$, it can be proved that all the lattices of the form $(\Gamma\oplus\langle 4x\rangle)'$ are in the same genus, and the same holds for all the lattices of the form $(\Gamma\oplus\langle 4x\rangle)^\star$; however, since $\lambda(A_{(\Gamma\oplus\langle 4x\rangle)'})=(2,2,2,2,4,4,4x)$, and $\lambda(A_{(\Gamma\oplus\langle 4x\rangle)^\star})=(2,2,2,2,2,2,4x)$, actually $(\Gamma\oplus\langle 4x\rangle)'$ and $(\Gamma\oplus\langle 4x\rangle)^\star$ are unique in their genus by Proposition \ref{condizione1}.
Furthermore, they admit a unique primitive embedding in $\Lambda_{\mathrm{K3}}$, as we can apply Proposition \ref{condizione1} to the corresponding trascendental lattices, both of signature $(2,5)$ and length 7, $T'=((\Gamma\oplus\langle 4x\rangle)')^{\perp}$, $T^{\star}=((\Gamma\oplus\langle 4x\rangle)^{\star})^{\perp}$: indeed $\lambda(A_{T'})=\lambda(A_{(\Gamma\oplus\langle 4x\rangle)'})$, and
$\lambda(A_{T^{\star}})=\lambda(A_{(\Gamma\oplus\langle 4x\rangle)^\star})$.
\end{proof}

\begin{theorem}\label{relationsZ}
Let $\tau$ be a symplectic automorphism of order 4 on a projective K3 surface $X$ such that $rk(NS(X))=15$, and consider $\tilde Z$ that is the resolution of the quotient $X/\tau^2$: the following table describes the correspondence between $NS(X)$ and $NS(\tilde Z)$. The primitive classes $\hat L$ in $NS(\tilde Z)$ that generate the sublattices $\langle nd\rangle$ as stated are indicated in curly brackets.
\begin{longtable}{|c|c|c c|c|}
\hline
\multicolumn{2}{|c|}{$NS(X)$\Tstrut} 				                     & \multicolumn{2}{c|}{$NS(\tilde Z)$ \Tstrut}  &$\hat L$ glues to $N$\\ [6pt]
\hline
$\forall d$\Tstrut	& $\Omega_4\oplus\langle 2d\rangle$\Tstrut  
		  	& $(\langle 4d\rangle\oplus\Gamma)' $ & $ \{\hat L=\pi_{2*}L\}$ & Yes \\ [6pt]
\hline
\multirow{2}{*}{$d=_4 2$\Tstrut} & $(\Omega_4\oplus\langle 2d\rangle)'^{(1)}$\Tstrut 
& \multirow{2}{*}{$(\langle 4d\rangle\oplus\Gamma)^\star$\Tstrut} &  \multirow{2}{*}{$ \big\{\hat L=\pi_{2*}L\big\} $\Tstrut} & \multirow{2}{*}{Yes\Tstrut} \\[6pt]
\ & $(\Omega_4\oplus\langle 2d\rangle)'^{(2)}$\Tstrut &\ &\  &\  \\ [6pt]
\hline
$d=_4 3$\Tstrut & $(\Omega_4\oplus\langle 2d\rangle)'$\Tstrut  
& $(\langle 4d\rangle\oplus\Gamma)^\star $ & $ \{\hat L=\pi_{2*}L\}$ & Yes \\ [6pt]
\hline
\multirow{2}{*}{$d=_4 0$\Tstrut} & $(\Omega_4\oplus\langle 2d\rangle)'$\Tstrut 
& \multirow{2}{*}{$(\langle d\rangle\oplus\Gamma)' $\Tstrut} & \multirow{2}{*}{ $ \big\{\hat L=\frac{\pi_{2*}L}{2}\big\} $\Tstrut} & \multirow{2}{*}{No\Tstrut} \\[6pt]
\ & $(\Omega_4\oplus\langle 2d\rangle)^{\star}$\Tstrut &\ &\  &\   \\ [6pt]
\hline
\end{longtable}
\end{theorem}
\begin{proof}
Recall from Theorem \ref{NSXordine2} the possible Néron-Severi groups of $X$. We use the map $\pi_{2*}$ (see Section \ref{sec:maps}) to compute $\hat L$ for each of the $L$ in Example \ref{esempiL}, and we check their eventual gluing to $N$ following Section \ref{sec:resolutionZ}. \\
Fix $d=4h+2=4(k-1)+2$, and consider the ample classes $L^{(1)}_{2,2}(h)$, $L^{(2)}_{2,2}(k)$ of $X$ that generate the two non isomorphic overlattices of index 2 of $\Omega_4\oplus\langle 2d\rangle$: denote these lattices $NS(X)^{(1)}$ and $NS(X)^{(2)}$. Now take $\tilde Z$ the resolution of $X/\tau^2$: from the previous theorem, we have $NS(\tilde Z)^{(1)}=\langle\Gamma,\hat L^{(1)}_{2,2}(h)\rangle\simeq NS(\tilde Z)^{(2)}=\langle\Gamma,\hat L^{(2)}_{2,2}(k)\rangle$.
Therefore, for $d=_4 2$ there is a 2-1 correspondence between $(\Omega_4\oplus\langle 2d\rangle)'$-polarized families of $X$ and $(\Gamma\oplus\langle 4d\rangle)^\star$-polarized families of $\tilde Z$. Similar considerations apply to $d=_4 0$. 
\end{proof}

\section{Projective models}

Given a nef and big divisor $L$ on $X$, there is a natural map $\phi_{|L|}: X\rightarrow\mathbb P(H^0(X,L)^*)\simeq \mathbb P^n$,
with $n=L^2/2+1$. Any automorphism $\sigma$ of $X$ that preserves $L$ induces an action on $H^0(X,L)$: in particular, if $\sigma$ is finite of order $m$, we can split $H^0(X,L)$ in eigenspaces corresponding to the $m$-roots of unity.
Thus, considering the action of a symplectic automorphism $\tau$ of order 4 on a K3 surface $X$, we have
\[H^0(X,L)=V_1\oplus V_i \oplus V_{-1}\oplus V_{-i}=W_+\oplus W_-\]
where $V_\bullet$ are the eigenspaces relative to the action of $\tau^*$, and $W_\bullet$ are relative to $(\tau^2)^*$, so that $W_+=V_1\oplus V_{-1}$, and $W_-=V_i\oplus V_{-i}$.

\subsection{Eigenspaces of $\tau^*$ and classes in $NS(\tilde Y)$}\label{sec:eigenspaces}

The purpose of this section is to prove the following proposition: 
\begin{proposition}[see \cite{VGS}, Prop. 2.7 and \cite{GP}, Thm. 5.6]\label{prop:autospazi}
\[H^0(X,L)=\pi_4^*H^0(\tilde Y, D_1)\oplus\pi_4^*H^0(\tilde Y, D_2)\oplus\pi_4^*H^0(\tilde Y, D_3)\oplus\pi_4^*H^0(\tilde Y, D_4)\]
for some divisors $D_1,\dots , D_4\in NS(\tilde Y)$, and  every $\pi_4^*H^0(\tilde Y, D_i)$ corresponds to one of the eigenspaces for the action of $\tau^*$ on $H^0(X,L)$. 
\end{proposition}

We fix some notation to begin with: consider the following elements in $M^*$, for ${i=1,\dots , 4}$, ${j=1,2}$ (see also Section \ref{Y}):
\[\alpha^i=\frac{3m^i_1+2m^i_2+m^i_3}{4},\quad \beta^i=\frac{m^i_1+2m^i_2+m^i_3}{2}, \quad \gamma^i=\frac{m^i_1+2m^i_2+3m^i_3}{4},\quad \delta^j=\frac{\tilde m^j}{2};\]
notice that $(\alpha^i)^2=(\gamma^i)^2=-3/4,\ (\beta^i)^2=-1,\ (\delta^j)^2=-1/2$ with respect to the intersection form of $M$ extended $\mathbb Q$-linearly to $M^*$.
\begin{itemize}[--]
\item Consider $L_0(d)$; depending on the value of $d$ mod $4$, we define $D_1,\dots , D_4$ as follows:
\end{itemize}
\begin{center}\begin{tabular}{|m{0.1\textwidth}<{\centering}|m{0.4\textwidth}<{\centering}|m{0.4\textwidth}<{\centering}|}
\hline
$L_0(d) \TBstrut$ &{ $d=_4 0$ \TBstrut }                      &{ $d=_4 1$  \TBstrut} \\  \hline
$D_1\Tstrut $ & $\pi_{4*}L_0/4-\gamma^2-\gamma^4-\delta^2$ \Tstrut & $\pi_{4*}L_0/4-\gamma^2-\alpha^3-\delta^1-\delta^2$\Tstrut  \\
$D_2\Tstrut $ & $\pi_{4*}L_0/4-\alpha^1-\alpha^3-\delta^1\Tstrut$  & $\pi_{4*}L_0/4-\alpha^1-\beta^3-\alpha^4$\Tstrut  \\
$D_3\Tstrut$ & $\pi_{4*}L_0/4-\beta^1-\alpha^2-\beta^3-\alpha^4-\delta^2\Tstrut$  & $\pi_{4*}L_0/4-\beta^1-\alpha^2-\gamma^3-\beta^4-\delta^1-\delta^2\Tstrut$\\
$D_4\TBstrut $ & $\pi_{4*}L_0/4-\gamma^1-\beta^2-\gamma^3-\beta^4-\delta^1$\TBstrut  & $\pi_{4*}L_0/4-\gamma^1-\beta^2-\gamma^4$ \TBstrut\\ \hline
\end{tabular}\end{center}
\begin{center}\begin{tabular}{|m{0.1\textwidth}<{\centering}|m{0.4\textwidth}<{\centering}|m{0.4\textwidth}<{\centering}|}
\hline
$L_0(d)\TBstrut$ &{ $d=_4 2$ \TBstrut}                      &{ $d=_4 3$ \TBstrut} \Tstrut \\  \hline
$D_1\Tstrut$ & $\pi_{4*}L_0/4-\gamma^2-\beta^3-\alpha^4-\delta^2\Tstrut$ & $\pi_{4*}L_0/4-\gamma^2-\gamma^3-\beta^4-\delta^1-\delta^2$\Tstrut \\
$D_2\Tstrut$ & $\pi_{4*}L_0/4-\alpha^1-\gamma^3-\beta^4-\delta^1\Tstrut$ & $\pi_{4*}L_0/4-\alpha^1-\gamma^4$\Tstrut  \\
$D_3\Tstrut$ & $\pi_{4*}L_0/4-\beta^1-\alpha^2-\gamma^4-\delta^2\Tstrut$ & $\pi_{4*}L_0/4-\beta^1-\alpha^2-\alpha^3-\delta^1-\delta^2$ \Tstrut  \\
$D_4\TBstrut$ & $\pi_{4*}L_0/4-\gamma^1-\beta^2-\alpha^3-\delta^1\TBstrut$ & $\pi_{4*}L_0/4-\gamma^1-\beta^2-\beta^3-\alpha^4$ \TBstrut\\ \hline
\end{tabular}\end{center}
\begin{itemize}[--]
\item Consider $L_{2,0}(h)$, whose square is $2d=8(h-1)$; depending on the value of $h$ mod $2$, we define $D_1,\dots , D_4$ as follows.
\end{itemize}
\begin{center}\begin{tabular}{|m{0.1\textwidth}<{\centering}|m{0.4\textwidth}<{\centering}|m{0.4\textwidth}<{\centering}|}
\hline
$L_{2,0}(h) \TBstrut$ &{ $h=_2 0$ \TBstrut }                      &{ $h=_2 1$  \TBstrut} \\  \hline
$D_1\Tstrut $ & $\pi_{4*}L_{2,0}/4-\beta^2-\beta^4 \Tstrut$ & $\pi_{4*}L_{2,0}/4-\beta^2-\beta^3$\Tstrut  \\
$D_2\Tstrut $ & $\pi_{4*}L_{2,0}/4-\alpha^1-\gamma^2-\alpha^3-\gamma^4-\delta^1-\delta^2\Tstrut$  & $\pi_{4*}L_{2,0}/4-\alpha^1-\gamma^2-\gamma^3-\alpha^4-\delta^1-\delta^2$\Tstrut  \\
$D_3\Tstrut$ & $\pi_{4*}L_{2,0}/4-\beta^1-\beta^3\Tstrut$  & $\pi_{4*}L_{2,0}/4-\beta^1-\beta^4\Tstrut$\\
$D_4 \TBstrut$ & $\pi_{4*}L_{2,0}/4-\gamma^1-\alpha^2-\gamma^3-\alpha^4-\delta^1-\delta^2\TBstrut$ & $\pi_{4*}L_{2,0}/4-\gamma^1-\alpha^2-\alpha^3-\gamma^4-\delta^1-\delta^2\TBstrut$ \\ \hline
\end{tabular}\end{center}
\begin{itemize}[--]
\item Consider $L^{(j)}_{2,2}(h)$, $j=1,2$; recall that $L^{(1)}_{2,2}(h)^2=8h+4$, while $L^{(2)}_{2,2}(h)^2=8h-4$: thus, any value of $d=_4 2$ can be realized \emph{both} with $h$ even \emph{and} $h$ odd, using one between $L^{(1)}_{2,2}, L^{(2)}_{2,2}$ alternatively, giving two non-isomorphic cases. 
\end{itemize}
\begin{center}\begin{tabular}{|m{0.1\textwidth}<{\centering}|m{0.4\textwidth}<{\centering}|m{0.4\textwidth}<{\centering}|}
\hline
$L^{(j)}_{2,2}(h) \TBstrut$ &{ $h=_2 0$ \TBstrut }                      &{ $h=_2 1$  \TBstrut} \\  \hline
$D_1\Tstrut $ & $\pi_{4*}L^{(j)}_{2,2}/4-\beta^3-\beta^4-\delta^1-\delta^2 \Tstrut$ & $\pi_{4*}L^{(j)}_{2,2}/4-\delta^1-\delta^2$\Tstrut  \\
$D_2\Tstrut $ & $\pi_{4*}L^{(j)}_{2,2}/4-\alpha^1-\alpha^2-\gamma^3-\gamma^4\Tstrut$  & $\pi_{4*}L^{(j)}_{2,2}/4-\alpha^1-\alpha^2-\alpha^3-\alpha^4$\Tstrut  \\
$D_3\Tstrut$ & $\pi_{4*}L^{(j)}_{2,2}/4-\beta^1-\beta^2-\delta^1-\delta^2\Tstrut$  & $\pi_{4*}L^{(j)}_{2,2}/4-\beta^1-\beta^2-\beta^3-\beta^4-\delta^1-\delta^2\Tstrut$\\
$D_4 \TBstrut$ & $\pi_{4*}L^{(j)}_{2,2}/4-\gamma^1-\gamma^2-\alpha^3-\alpha^4\TBstrut$ & $\pi_{4*}L^{(j)}_{2,2}/4-\gamma^1-\gamma^2-\gamma^3-\gamma^4\TBstrut$ \\ \hline
\end{tabular}\end{center}
\begin{itemize}[--]
\item Consider $L_{2,3}(h)$, whose square is $2d=2(4h+3)$; depending on the value of $h$ mod $2$, we define $D_1,\dots , D_4$ as follows.
\end{itemize}
\begin{center}\begin{tabular}{|m{0.1\textwidth}<{\centering}|m{0.4\textwidth}<{\centering}|m{0.4\textwidth}<{\centering}|}
\hline
$L_{2,3}(h) \TBstrut$ &{ $h=_2 0$ \TBstrut }                      &{ $h=_2 1$  \TBstrut} \\  \hline
$D_1\Tstrut $ & $\pi_{4*}L_{2,3}/4-\beta^4-\delta^2 \Tstrut$ & $\pi_{4*}L_{2,3}/4-\beta^3-\delta^2$\Tstrut  \\
$D_2\Tstrut $ & $\pi_{4*}L_{2,3}/4-\alpha^1-\alpha^2-\alpha^3-\gamma^4-\delta^1\Tstrut$  & $\pi_{4*}L_{2,3}/4-\alpha^1-\alpha^2-\gamma^3-\alpha^4-\delta^1$\Tstrut  \\
$D_3\Tstrut$ & $\pi_{4*}L_{2,3}/4-\beta^1-\beta^2-\beta^3-\delta^2\Tstrut$  & $\pi_{4*}L_{2,3}/4-\beta^1-\beta^2-\beta^4-\delta^2\Tstrut$\\
$D_4 \TBstrut$ & $\pi_{4*}L_{2,3}/4-\gamma^1-\gamma^2-\gamma^3-\alpha^4-\delta^1\TBstrut$ & $\pi_{4*}L_{2,3}/4-\gamma^1-\gamma^2-\alpha^3-\gamma^4-\delta^1\TBstrut$ \\ \hline
\end{tabular}\end{center}
\begin{itemize}[--]
\item Consider  $L_{4,j}$ for $j=0,4,8,12$; in this case $\pi_{4*}L_{4,j}/4$ is primitive in $NS(\tilde Y)$, and we can define $D_1,\dots , D_4$ simultaneously for any $j$ and any value of $h$, as follows:
\end{itemize}
\begin{center}\begin{tabular}{|m{0.1\textwidth}<{\centering}|m{0.4\textwidth}<{\centering}|}
\hline
$L_{4,j}(h) \TBstrut$ &{ any $h$ \TBstrut }                       \\  \hline
$D_1\Tstrut $ & $\pi_{4*}L_{4,j}/4$  \\
$D_2\Tstrut $ & $\pi_{4*}L_{4,j}/4-\alpha^1-\alpha^2-\alpha^3-\alpha^4-\delta^1-\delta^2\Tstrut$  \\
$D_3\Tstrut$ & $\pi_{4*}L_{4,j}/4-\beta^1-\beta^2-\beta^3-\beta^4\Tstrut$  \\
$D_4 \TBstrut$ & $\pi_{4*}L_{4,j}/4-\gamma^1-\gamma^2-\gamma^3-\gamma^4-\delta^1-\delta^2\TBstrut$  \\ \hline
\end{tabular}\end{center}

\begin{proof}[Proof of Proposition \ref{prop:autospazi}]
Consider $L$ any of the ample divisors of $X$ presented in Example \ref{esempiL}, and the corresponding $D_1,\dots , D_4$ as in the tables above: notice that for every $i$ the relation $\pi_4^*(D_i)=L$ is satisfied (since $\pi_4^*M=0$): therefore we always have
$\pi_4^*H^0(\tilde Y, D_i)\subset H^0(X,L)$.
Moreover, $\pi_4^*H^0(\tilde Y, D_i(L))$ is all contained in one of the eigenspaces $V_\bullet(L)$ (the sections of $D_i$ are in fact well defined on the
quotient surface $\tilde Y$), and for $i\neq j$ $\pi_4^*H^0(\tilde Y, D_i(L))$ and $\pi_4^*H^0(\tilde Y, D_j(L))$ are in different eigenspaces, since $D_i(L)$ and $D_j(L)$ intersect differently the exceptional lattice for $i\neq j$.\\
It remains to show that $H^0(X,L)=\bigoplus_{i=1}^4\pi_4^*H^0(\tilde Y, D_i)$: to do this,
it is enough to compute the Euler characteristics $\chi(D_i)=D_i^2/2+2$, and check that \[d+2=\chi(L)=\sum_i\chi(D_i).\] 
The results are displayed in the following table.

\begin{center}
\captionof{table}{Euler characteristics\label{Tab:dim}}
\begin{tabular}{|m{0.1\textwidth}<{\centering}|m{0.03\textwidth}<{\centering}|m{0.07\textwidth}<{\centering}|m{0.15\textwidth}<{\centering}m{0.15\textwidth}<{\centering}m{0.15\textwidth}<{\centering}m{0.15\textwidth}<{\centering}|}
\cline{2-7}
\nocell{1} &no.\TBstrut & $L\TBstrut$ & $\chi(D_1)\TBstrut$  & $\chi(D_2)\TBstrut$ & $\chi(D_3)\TBstrut$ & $\chi(D_4)$\TBstrut \\ 
\hline
$d=_4 1\TBstrut$ &1\Tstrut & $L_0\TBstrut$   & $(d+3)/4\TBstrut$  & $(d+3)/4\TBstrut$ & $(d-1)/4\TBstrut$ & $(d+3)/4$\TBstrut \\
\hline
\multirow{3}{*}{$d=_4 2$\TTstrut} 
&2\Tstrut & $L_0\Tstrut$   & $(d+2)/4\Tstrut$  & $(d+2)/4\Tstrut$ & $(d+2)/4\Tstrut$ & $(d+2)/4$\Tstrut \\
&3\Tstrut & $L^{(i)}_{2,2}\Tstrut$  & $(d+2)/4\Tstrut$  & $(d+2)/4\Tstrut$ & $(d+2)/4\Tstrut$ & $(d+2)/4$\Tstrut \\
&4\Tstrut & $L^{(j)}_{2,2}\TBstrut$  & $(d+6)/4\TBstrut$  & $(d+2)/4\TBstrut$ & $(d-2)/4\TBstrut$ & $(d+2)/4$\TBstrut \\
\hline
\multirow{2}{*}{$d=_4 3$\Tstrut} 
&5\Tstrut & $L_0\Tstrut$   & $(d+1)/4\Tstrut$  & $(d+5)/4\Tstrut$ & $(d+1)/4\Tstrut$ & $(d+1)/4$\Tstrut \\
&6\Tstrut & $L_{2,3}\TBstrut$   & $(d+5)/4\TBstrut$  & $(d+1)/4\TBstrut$ & $(d+1)/4\TBstrut$ & $(d+1)/4$\TBstrut \\
\hline
\multirow{3}{1.5cm}{\centering{$d=_4 0\TTstrut$}} 
&7\Tstrut & $L_0\Tstrut$   & $d/4+1\Tstrut$  & $d/4+1\Tstrut$ & $d/4\Tstrut$ & $d/4$\Tstrut \\
&8\Tstrut & $L_{2,0}\Tstrut$  & $d/4+1\Tstrut$  & $d/4\Tstrut$ & $d/4+1\Tstrut$ & $d/4$\Tstrut \\
&9\TBstrut & $L_{4,j}\TBstrut$ & $d/4+2\TBstrut$  & $d/4\TBstrut$ & $d/4\TBstrut$ & $d/4$\TBstrut \\
\hline
\end{tabular}\end{center}
\end{proof}

\subsection{Eigenspaces of $\tau^{2*}$}

The automorphism $\tau^2$ is a symplectic involution on $X$: the following proposition describes the eigenspaces of a general symplectic involution.
\begin{proposition}[\cite{VGS}, Prop. 2.7]
Let $\iota$ be a symplectic involution on a K3 surface $X$ such that $rk(NS(X))=9$, let $Z=\widetilde{X/\iota}$ the resolution of the quotient surface, and let $\pi:X\dashrightarrow Z$ the induced rational map. Let $L$ be an ample divisor on $X$, such that $L^2=2d$, that generates $\Omega_2^{\perp_{NS(X)}}$. Then $H^0(X, L)\simeq \pi^*H^0(Z, E_1)\oplus \pi^*H^0(Z, E_2)$, with $E_1$, $E_2$ described as follows, for suitable numbering $n_1,\dots , n_8$ of the exceptional curves of $Z$:
\begin{enumerate}
\item if $NS(X)=\Omega_2\oplus\langle L\rangle$, and $d=_2 0$, then $E_1=\pi_*L/2-(n_1+n_2+n_3+n_4)/2$, $E_2=\pi_*L/2-(n_5+n_6+n_7+n_8)/2$;
\item  if $NS(X)=\Omega_2\oplus\langle L\rangle$, and $d=_2 1$, then $E_1=\pi_*L/2-(n_1+n_2)/2$, $E_2=\pi_*L/2-(n_3+n_4+n_5+n_6+n_7+n_8)/2$;
\item if $NS(X)=(\Omega_2\oplus\langle L\rangle)'$ (this case occurs only if $d=_2 0$), then $E_1=\pi_*L/2$, $E_2=\pi_*L/2-\sum_{i=1}^8n_i/2$.
\end{enumerate}
\end{proposition}

If $X$ admits an automorphism $\tau$ of order 4 and $\tilde Z$ is the minimal resolution of $X/\tau^2$, taking $L$ ample that generates $\Omega_4^{\perp_{NS(X)}}$ we have the following relation between the eigenspaces of $\tau^*$ and $\tau^{2*}$:
\[H^0(X,L)=\bigoplus_{i=1}^4\pi_4^*H^0(\tilde Y, D_i)=\pi_2^*H^0(\tilde Z, E_1)\oplus \pi_2^*H^0(\tilde Z, E_2);\]
the Nef divisors $E_1, E_2\in NS(\tilde Z)$ that satisfy this equality for the examples of ample classes introduced in Example \ref{esempiL} are defined in the following tables, with the exceptional curves numbered as in Sections \ref{sec:resolutionZ}, \ref{Y}:

\begin{tabular}{|m{0.08\textwidth}<{\centering}|m{0.4\textwidth}<{\centering}|m{0.43\textwidth}<{\centering}|}
\hline
$L_0(d) \TBstrut$ &{ $d=_2 0$ \TBstrut }                      &{ $d=_2 1$  \TBstrut} \\  \hline
$E_1\Tstrut $ & $\pi_{2*}L_0/2-(n_5+n_6+n_7+n_8)/2 \Tstrut$ & $\pi_{2*}L_0/2-(n_2+n_3+n_4+n_5+n_6+n_8)/2$\Tstrut  \\
$E_2\TBstrut $ & $\pi_{2*}L_0/2-(n_1+n_2+n_3+n_4)/2\TBstrut$  & $\pi_{2*}L_0/2-(n_1+n_7)/2$\TBstrut  \\
 \hline
\end{tabular}

\begin{tabular}{|m{0.08\textwidth}<{\centering}|m{0.30\textwidth}<{\centering}|}
\hline
$L_{2,0}(h) \TBstrut$ &{ any $h$ \TBstrut }                       \\  \hline
$E_1\Tstrut $ & $\pi_{2*}L_{2,0}/2\Tstrut$ \\
$E_2\TBstrut $ & $\pi_{2*}L_{2,0}/2-\sum_{i=1}^8 n_i/2\TBstrut$ \\
 \hline
\end{tabular}\  
\begin{tabular}{|m{0.08\textwidth}<{\centering}|m{0.4137\textwidth}<{\centering}|}
\hline
$L^{(j)}_{2,2}(h) \TBstrut$ &{ any $h,\ j=1,2$ \TBstrut }                       \\  \hline
$E_1\Tstrut $ & $\pi_{2*}L^{(j)}_{2,2}/2-(n_3+n_4+n_5+n_8)/2 \Tstrut$ \\
$E_2\TBstrut $ & $\pi_{2*}L^{(j)}_{2,2}/2-(n_1+n_2+n_6+n_7)/2\TBstrut$ \\
 \hline
\end{tabular}

\begin{tabular}{|m{0.08\textwidth}<{\centering}|m{0.44\textwidth}<{\centering}|}
\hline
$L_{2,3}(h) \TBstrut$ &{ any $h$ \TBstrut }                       \\  \hline
$E_1\Tstrut $ & $\pi_{2*}L_{2,3}/2-(n_5+n_8)/2 \Tstrut$ \\
$E_2\TBstrut $ & $\pi_{2*}L_{2,3}/2-(n_1+n_2+n_3+n_4+n_6+n_7)/2\TBstrut$ \\
 \hline
\end{tabular}\ 
\begin{tabular}{|m{0.08
\textwidth}<{\centering}|m{0.2737\textwidth}<{\centering}|}
\hline
$L_{4,j}(h) \TBstrut$ &{ any $h,\ j=0,4,8,12$ \TBstrut }                       \\  \hline
$E_1\Tstrut $ & $\pi_{2*}L_{4,j}/2\Tstrut$ \\
$E_2\TBstrut $ & $\pi_{2*}L_{4,j}/2-\sum_{i=1}^8 n_i/2\TBstrut$ \\
 \hline
\end{tabular}

\begin{proposition}
It holds $\pi_2^*H^0(\tilde Z, E_1)=\pi_4^*H^0(\tilde Y, D_1)\oplus \pi_4^*H^0(\tilde Y, D_3)$, while \\$\pi_2^*H^0(\tilde Z, E_2)=\pi_4^*H^0(\tilde Y, D_2)\oplus \pi_4^*H^0(\tilde Y, D_4)$.
\end{proposition}
\begin{proof}
It's easy to see that the dimensions agree. 
Moreover, notice that if $E_1$ intersects positively $n_i\in\{n_1,n_2,n_6,n_7\}$ (classes fixed by $\hat\tau^*$), then $\phi_{|E_1|}(n_i)$ is a curve $\mathcal C$ in $\phi_{|E_1|}(\tilde Z)$; consider now the induced automorphism $\hat\tau$ on $\phi_{|E_1|}(\tilde Z)$: it fixes two points $p_1,p_2$ on $\mathcal C$, each belonging to one (or the other) of the eigenspaces for the action of $\hat\tau^*$ on $H^0(\tilde Z, E_1)$, that are
\[H^0(\tilde Z, E_1)=\widehat{\pi_2}^*H^0(\widetilde{\tilde Z/\hat\tau}, F_1)\oplus\widehat{\pi_2}^* H^0(\widetilde{\tilde Z/\hat\tau}, F_2)\]
for some divisors $F_1,F_2$ of $\widetilde{\tilde Z/\hat\tau}$. Therefore, $F_1+F_2$ intersects positively the two curves $\mathcal C_1, \mathcal C_2$, that resolve the singular points image of $p_1,p_2$ in $\tilde Z/\hat\tau$ .
If $E_1$ intersects trivially $n_i$, then $\phi_{|E_1|}(n_i)$ is a point $p$ in $\phi_{|E_1|}\tilde Z$, which is fixed by $\hat\tau$ and thus belongs to an eigenspace: its image in $\tilde Z/\hat\tau$ is resolved by a curve $\mathcal C_p$, that is intersected positively by either $F_1$ or $F_2$, and so by their sum.
A similar argument can be applied for $n_i\in\{n_3,n_4,n_5,n_8\}$, thus proving that how $F^{(i)}_1+F^{(i)}_2$ intersects the exceptional lattice of the resolution of ${\tilde Z/\hat\tau}$ depends on how $E_i$ intersects that of $\tilde Z$.\\
Since the surfaces $\widetilde{\tilde Z/\hat\tau}$ and $\tilde Y$ are isomorphic (see Rem. \ref{YZisomorfi}), we have
\[H^0(X, L)=\bigoplus_{i=1,2}\pi_4^*H^0(\widetilde{\tilde Z/\hat\tau}, F^{(i)}_1)\oplus\pi_4^* H^0(\widetilde{\tilde Z/\hat\tau}, F^{(i)}_2)\]
and a correspondence between each $H^0(\tilde Y, D_j)$ and one of the $H^0(\widetilde{\tilde Z/\hat\tau}, F^{(i)}_k)$. \\
The fact that this correspondence is exactly as stated comes from a comparison between how the $E_i$ and the $D_j$ intersect the exceptional lattices of $\tilde Z$ and $\tilde Y$ respectively. 
\end{proof}

\subsection{Examples in $\mathbb P^3$} 

There are three families of K3 surfaces polarized with an ample class $L$ such that $L^2=4$, corresponding to \textbf{no. 2}, \textbf{no. 3}, \textbf{no. 4} of Table \ref{Tab:dim}: since it holds $\chi(D_i)=h^0(D_i)$, as we proved in Proposition \ref{prop:autospazi}, we can read from Table \ref{Tab:dim} the dimension of the eigenspaces of the action induced by $\tau$. The first two cases give eigenspaces of the same dimension but, as we will see below, \textbf{no. 2} provides a model of $X$ as smooth quartic in $\mathbb P^3$, while \textbf{no. 3} as double cover of $\mathbb P^1\times\mathbb P^1$. Recall from Definition \ref{def:latticepolarized} that the projective dimension of each of the families of $X$ is ${5=20-(rk(\Omega_4)+1)}$.

\textbf{no. 2}: The divisor $L_0(2)$ has square 4, and in $NS(X)=\Omega_4\oplus\mathbb ZL$ there exists no class $E$ such that $E^2=0$ and $EL_0(2)=2$: therefore (see \cite{SaintDonat}, Thm. 5.2), $\phi_{|L_0(2)|}: X \hookrightarrow\mathbb P^3$ is an embedding of $X$ in $\mathbb P^3$ as a quartic surface. Consider the automorphism of $\mathbb P^3$:
\[\psi_2: (x_0:x_1:x_2:x_3)\mapsto (x_0:ix_1:-x_2:-ix_3);\]
 quartic surfaces of the form
\[Q_2:\ x_0^3x_2+x_0^2(x_1^2+x_3^2)+x_0(x_2^3+\delta x_1x_2x_3)+x_2^2(\varepsilon x_1^2+\zeta x_3^2)+\eta x_1^3x_3+\theta x_1x_3^3=0\]
are invariant under the action of $\psi_2$, and they depend on 5 projective parameters up to projectivities of the form ${(x_0:x_1:x_2:x_3)\mapsto (x_0:ax_1:bx_2:cx_3)}$, which commute with $\psi_2$. Moreover, $Q_2$ contains exactly 8 points fixed by $\psi_2^2$, of wich 4 are fixed also by $\psi_2$; we have therefore \[\phi_{|L_0(2)|}: X \xrightarrow{\simeq} Q_2\subset\mathbb P^3.\]
To find models for the quotient surfaces, as in (\cite{VGS}, \S 3.4) we consider the map given by the degree 2 invariants under the action of $\psi_2^2$
\[(x_0:x_1:x_2:x_3)\mapsto(x_0^2:x_1^2:x_2^2:x_3^2:x_0x_2:x_1x_3)=(z_0:z_1:z_2:z_3:z_4:z_5);\]
then the surface $Q_2$ maps to the complete intersection of quadrics in $\mathbb P^5$
\begin{equation*} R_2:
    \begin{cases}
      z_0z_2-z_4^2=0\\
      z_5^2-z_1z_3=0\\
	 z_0z_4+z_0(\alpha z_1+\beta z_3)+z_4(\gamma z_2+\delta z_5)+z_2(\varepsilon z_1+\zeta z_3)+z_5(\eta z_1+\theta z_3)=0
    \end{cases}
\end{equation*}
which is a projective model for $Q_2/\psi_2^2$. Since $\hat L_0(2)=\pi_{2*}L_0(2)$ has self-intersection 8, it holds
\[\phi_{|\hat L_0(2)|}:Z\xrightarrow{\simeq} R_2\subset\mathbb P^5.\]
Now, the automorphism induced by $\psi_2$ on $\mathbb P^5$ is
\[\hat\psi_2:(z_0:z_1:z_2:z_3:z_4:z_5)\mapsto(z_0:-z_1:z_2:-z_3:-z_4:z_5):\]
since the surface $R_2$ has the same form as in (\cite{VGS}, \S 3.6), then the quotient of $R_2$ under the action of $\hat\psi_2$ is described by a complete intersection in $\mathbb P^2_{(z_0:z_2:z_5)}\times\mathbb P^2_{(z_1:z_3:z_4)}$ of two polynomials of bidegree respectively $(2,2)$, $(1,1)$, that is
\begin{equation*} S_2:
    \begin{cases}
      z_0z_2z_1z_3-z_4^2z_5^2=0\\
	 z_0z_4+z_0(\alpha z_1+\beta z_3)+z_4(\gamma z_2+\delta z_5)+z_2(\varepsilon z_1+\zeta z_3)+z_5(\eta z_1+\theta z_3)=0.
    \end{cases}
\end{equation*}

\textbf{no. 3}: The eigenspaces associated to the action of $\tau$ on $H^0(X,L^{(1)}_{2,2}(0))$ have all the same dimension as in the previous case. However, as in (\cite{VGS}, \S 3.5), the divisor $L^{(1)}_{2,2}(0)$ is ample but not very ample: indeed (see \cite{SaintDonat}, Thm. 5.2), we have $L^{(1)}_{2,2}(0)=H_1+H_2$ with
\[H_1=\frac{L_0(0)+\rho+\sigma}{2}, \ H_2=\frac{L_0(0)+\rho-\sigma}{2};\quad\langle H_1, H_2\rangle=\begin{bmatrix}0 & 2\\ 2 & 0\end{bmatrix};\quad \tau^*(H_1)=H_2.\]
Hence
\[\phi_{|L^{(1)}_{2,2}(0)|}=\phi_{|H_1+H_2|}: X\xrightarrow{2:1} \mathbb P^1\times\mathbb P^1\]
is a double cover ramified along a curve of bidegree $(4,4)$ invariant for the automorphism of $\mathbb P^1\times \mathbb P^1$
\[\overline\psi_3: (x_0:x_1)(y_0:y_1)\mapsto(y_0:iy_1)(x_0:ix_1);\]
curves of this type depend on 5 projective parameters when taking into account the action of the group of projectivities of the form ${(x_0:x_1)(y_0:y_1)}\mapsto {(x_0:ax_1)(y_0:ay_1)}$, which are the only ones that commute with $\overline\psi_3$. We embed $\mathbb P^1\times \mathbb P^1$ in $\mathbb P^3$ via the Segre map
\[(x_0:x_1)(y_0:y_1)\mapsto(x_0y_0:x_0y_1:x_1y_0:x_1y_1)=(z_0:z_1:z_2:z_3):\]
now $X$ is a double cover of the quadric surface $z_0z_3=z_1z_2$ ramified along a curve of degree $4$ invariant for the automorphism $\psi_3$ of $\mathbb P^3$ induced by $\overline\psi_3$ via the Segre map,
\[\psi_3: (z_0:z_1:z_2:z_3)\mapsto (z_0:iz_2:iz_1:-z_3).\]
Hence, $X$ is described in $\mathbb P(2,1,1,1,1)$ by
\begin{equation*} Q_3:
    \begin{cases}
      z_0z_3=z_1z_2\\
      w^2=\alpha z_0^4+\beta z_0^2z_3^2+\gamma z_3^4+z_0z_3(\delta z_1^2+\varepsilon z_1z_2+\delta z_2^2)+\zeta z_1^4+\eta  z_1^2z_2^2+\zeta z_2^4;
    \end{cases}
\end{equation*}
notice that the fixed locus of $\psi_3$ on $\mathbb P^3$ is $\{(1:0:0:0),(0:0:0:1),(0:1:1:0)\}$: only the first two of these points belong to the branch curve, so to have 4 points fixed by $\psi_3$ on $Q_3$, the action induced by $\psi_3$ on $\mathbb P(2,1,1,1,1)$ has to be 
\[(w;z_0:z_1:z_2:z_3)\mapsto (w;z_0:iz_2:iz_1:-z_3).\]

To find a projective model of the quotient surface $Z$, we consider the degree $2$ invariants for the action of $\psi_3^2$, that form a projective space of dimension 6:
\[(w;z_0:z_1:z_2:z_3)\mapsto(w:z_0^2:z_1^2:z_2^2:z_3^2:z_0z_3:z_1z_2)=(w:a_0:a_1:a_2:a_3:a_4:a_5);\]
the surface $Q_3$ maps to
\begin{equation*} R_3:
    \begin{cases}
      a_0a_3=a_4^2\\
      a_1a_2=a_5^2\\
      a_4=a_5\\
      w^2=\alpha a_0^2+\beta a_0a_3+\gamma a_3^2+z_0z_3(\delta a_1+\varepsilon a_5+\delta a_2)+\zeta a_1^2+\eta  a_1a_2+\zeta a_2^2,
    \end{cases}
\end{equation*}
the complete intersection of 3 quadrics in the hyperplane defined by $a_4=a_5$ in $\mathbb P^6$.\\
Eliminate $a_5$, and change coordinates to
\[b_0=a_0,\quad b_1=a_1+a_2,\quad b_2=a_1-a_2,\quad b_3=a_3,\quad b_4=a_4;\]
then the automorphism induced by $\psi_3$ on $\mathbb P^5$ is
\[\hat\psi_3: (w:b_0:b_1:b_2:b_3:b_4)\mapsto (w:b_0:-b_1:b_2:b_3:-b_4),\]
and in the new coordinates we can write $R_3$ as
\begin{equation*} R_3:
    \begin{cases}
      b_1^2=b_0b_3\eqqcolon\rho_1(w,b_0,b_2,b_3)\\
      b_1b_4=b_2^2+4b_0b_3\eqqcolon\rho_2(w,b_0,b_2,b_3)\\
      b_4^2\eqqcolon\rho_3(w,b_0,b_2,b_3),
    \end{cases}
\end{equation*}
where $\rho_1,\rho_2,\rho_3$ are quadrics, similarly to (\cite{VGS}, \S 3.7): the quotient surface $S_3=R_3/\hat\psi_3$ is therefore the quartic surface $\rho_1\rho_3-\rho_2^2=0$ in $\mathbb P^3$.

\textbf{no. 4}: The eigenspaces associated to the action of $\tau$ on $H^0(X,L^{(2)}_{2,2}(1))$ have dimensions $2,1,0,1$: consider the automorphism of $\mathbb P^3$:
\[\psi_4: (x_0:x_1:x_2:x_3)\mapsto (x_0:x_1:ix_2:-ix_3);\]
quartic surfaces of the form
\[Q_4: f_4(x_0,x_1)+x_2x_3f_2(x_0,x_1)+\alpha x_2^4+\beta x_2^2x_3^2+\gamma x_3^4=0,\]
where $f_4, f_2$ are respectively homogeneous quartic and quadric polynomials, are invariant under the action of $\psi_4$, and they form a family of projective dimension 5 when taking into account the action of the group of projectivities of the form ${(x_0:x_1:x_2:x_3)\mapsto} (ax_0+bx_1:cx_0+dx_1:ex_2:x_3)$, that commute with $\psi_4$. Moreover, $Q_4$ contains exactly 4 points fixed by $\psi_4$, and 4 more fixed only by $\psi_4^2$. Therefore $L^{(2)}_{2,2}(1)$ is very ample, and
 \[\phi_{|L^{(2)}_{2,2}(1)|}: X \xrightarrow{\simeq} Q_4\subset\mathbb P^3.\]
Proceeding as in \textbf{no. 2}, we consider the map given by the degree 2 invariants under the action of $\psi_4^2$
\[(x_0:x_1:x_2:x_3)\mapsto(x_0^2:x_1^2:x_2^2:x_3^2:x_0x_1:x_2x_3)=(z_0:z_1:z_2:z_3:z_4:z_5);\]
then the quotient $Q_4/\psi^2|_{Q_4}$ is the complete intersection of quadrics in $\mathbb P^5$
\begin{equation*} R_4:
    \begin{cases}
      \rho_1: z_0z_1-z_4^2=0\\
      \rho_2: z_5^2-z_2z_3=0\\
	 \rho_3: \tilde f_4(z_0,z_1,z_4)+z_5\tilde f_2(z_0,z_1,z_4)+\alpha z_2^2+\beta z_5^2+\gamma z_3^2=0,
    \end{cases}
\end{equation*}
where $\tilde f_4,\tilde f_2$ are respectively homogeneous quadric and linear polynomials such that $\tilde f_4(x_0^2,x_1^2,x_0x_1)=f_4(x_0,x_1)$ (similarly $\tilde f_2$ and $f_2$).
The automorphism induced by $\psi_4$ on $\mathbb P^5$ is
\[\hat\psi_4:(z_0:z_1:z_2:z_3:z_4:z_5)\mapsto(z_0:z_1:-z_2:-z_3:z_4:z_5).\]
The surface $R_4$ is singular in 8 points, 4 obtained as $R_4\cap\{z_0=z_1=z_4=0\}$, which are not fixed by $\hat\psi_4$, and 4 obtained as $R_4\cap\{z_2=z_3=z_5=0\}$, which are fixed by $\hat\psi_4$.\\ 
To find the quotient $R_4/\hat\psi_4|_{R_4}$, we can consider the projection $\mathbb P^5\rightarrow\mathbb P^3$
\[\pi:(z_0:z_1:z_2:z_3:z_4:z_5)\mapsto(z_0:z_1:z_4:z_5)\]
from the line $\ell=(0:0:s:t:0:0)$. Notice that $\pi(\rho_1)=\rho_1$, and that for every $(\overline z_0:\overline z_1:\overline z_4:\overline z_5)\in\mathbb P^3$ we can compute its pre-image as
\begin{equation*}
    \begin{cases}
	 \overline z_5^2=st\\
	 \tilde f_4(\overline z_0,\overline z_1,\overline z_4)+\overline z_5\tilde f_2(\overline z_0,\overline z_1,\overline z_4)+\alpha s^2+\beta st+\gamma t^2=0;
    \end{cases}
\end{equation*}
setting $B=\beta\overline z_5^2+\overline z_5\tilde f_2+\tilde f_4$, this gives
\begin{equation*}
    \begin{cases}
	 s=\overline z_5^2/t\\
	 t^2=\frac{-B\pm\sqrt{B^2-4\alpha\gamma\overline z_5^4}}{2\gamma}.
    \end{cases}
\end{equation*}
There are generally 4 solutions, pairwise identified by the action of $\hat\psi_4$: we can therefore define a surface $S_4$ that completes the diagram
\[\xymatrix{
R_4 \ar[rd]_{2:1}^(.57){/\hat\psi_4} \ar[rr]^{4:1}_{\pi} &\ &{\rho_1} \\
\ &S_4 \ar[ru]_{2:1} &\ 
}\]
that is, the quotient $S_4=R_4/\hat\psi_4|_{R_4}$ is a double cover of the quadric $\rho_1\subset\mathbb P^3$ ramified over the curve defined by $B^2-4\alpha\gamma\overline z_5^4=0$, and thus is a K3 surface.

\end{document}